\documentclass[pdflatex,11pt]{article}

\usepackage{setspace}
\usepackage[left=1in,
            right=1in,
            top=1in,
            bottom=1in]{geometry}

\usepackage{authblk}
\usepackage{graphicx}%
\usepackage{multirow}%
\usepackage{amsmath,amssymb,amsfonts}%
\usepackage{amsthm}%
\usepackage{mathrsfs}%
\usepackage[title]{appendix}%
\usepackage{xcolor}%
\usepackage{textcomp}%
\usepackage{manyfoot}%
\usepackage{booktabs}%
\usepackage{algpseudocode}%
\usepackage{subfig}
\usepackage{caption}
\usepackage[ruled,vlined,linesnumbered]{algorithm2e}
\usepackage{enumitem}
\usepackage{todonotes}
\usepackage{comment}

\newtheorem{theorem}{Theorem}
\newtheorem{proposition}[theorem]{Proposition}%
\newtheorem{example}{Example}%
\newtheorem{remark}{Remark}%

\newcommand{\set}[2]{\left\{\left. #1 \;\right|\;\; #2 \right\}}
\newcommand{\abs}[1]{\lvert#1\rvert}
\newcommand{\vectornorm}[1]{\lvert{\lvert#1\rvert}\rvert}

\newcommand{\Umat}{A}
\newcommand{\Urhs}{b}
\newcommand{\R}{\mathbb{R}}
\newcommand{\K}{\mathcal{K}}
\newcommand{\subsetL}{\widetilde{\mathcal{L}}}
\renewcommand{\L}{\mathcal{L}}
\newcommand{\U}{\mathcal{U}}

\newcommand{\A}{\mathcal{A}}

\newcommand{\vone}{\boldsymbol{1}}

\newcommand{\feasibility}{\mathcal{Z}}
\newcommand{\costx}{f}
\newcommand{\costy}{g}
\newcommand{\nx}{{n_x}}
\newcommand{\ny}{{n_y}}
\newcommand{\nZ}{{n_z}}
\newcommand{\nedges}{{|E|}}
\newcommand{\nvertices}{{n}}
\newcommand{\epicost}{y^{TC}}

\newcommand{\E}{\mathcal{E}}
\newcommand{\X}{\mathcal{X}}

\newcommand{\Z}{\mathcal{Z}}
\renewcommand{\U}{\mathcal{U}}

\newcommand{\tU}{\widetilde{\mathcal{U}}}

\usepackage{accents}
\newcommand\ubar[1]{%
  \underaccent{\bar}{#1}}
  
 \newcommand{\USB}{\mathcal{U}_{\mathcal{S}}}
 \newcommand{\UE}{\mathcal{U}_{\mathcal{E}}}
 \newcommand{\Ubox}{\mathcal{U}_{box}}

 \newcommand{\primalCG}{\texttt{colgen}\xspace}
 \newcommand{\cutgen}{\texttt{cutgen}\xspace}
 \newcommand{\compact}{\texttt{dualization}\xspace}

 \newcommand{\norm}[1]{\|#1\|}
 \newcommand{\card}[1]{\left|#1\right|}
 
 \newcommand{\myqed}{\square}

\usepackage{url}
\usepackage{natbib}
\newcommand{\fnm}[1]{#1}
\newcommand{\sur}[1]{#1}
\newcommand{\email}[1]{\thanks{{#1}}}
\newcommand{\orgdiv}[1]{#1}
\newcommand{\orgname}[1]{#1}
\newcommand{\orgaddress}[1]{#1}
\newcommand{\city}[1]{#1}
\newcommand{\country}[1]{#1}
\newcommand{\postcode}[1]{#1}
\providecommand{\keywords}[2][Keywords]{\newline\newline\textbf{\textit{#1:} } #2}

\newtheorem{assumption}{Assumption}
\newtheorem{corollary}{Corollary} 
\newtheorem{lemma}[theorem]{Lemma}
\DeclareMathOperator{\diag}{diag}
\DeclareMathOperator{\cov}{cov}
\DeclareMathOperator{\dist}{dist}

\newcommand{\shim}[1]{{\color{black}#1}}

\renewcommand{\mp}[1]{{\color{black}#1}}
\renewcommand{\ng}[1]{{\color{black}#1}}
\newcommand{\ngtwo}[1]{{\color{black}#1}}
\newcommand{\ngthree}[1]{{\color{black}#1}}

\title{Smooth Uncertainty Sets:  
Dependence of Uncertain Parameters via a Simple Polyhedral Set}

\begin{document}
\onehalfspacing
\author[1]{\fnm{Noam} \sur{Goldberg}\email{goldnoam@bgu.ac.il}}
\author[2]{\fnm{Michael} \sur{Poss}\email{michael.poss@lirmm.fr}}
\author[3]{\fnm{Shimrit} \sur{Shtern}\email{shimrits@technion.ac.il}}
\affil[1]{\orgdiv{Department of Industrial Engineering \& Management}, \orgname{Ben-Gurion University of the Negev}, \city{Beer-Sheva}, \country{Israel}}
\affil[2]{\orgdiv{LIRMM}, \orgname{University of Montpellier, CNRS}, \orgaddress{\city{Montpellier},  \country{France}}}
\affil[3]{\orgdiv{Faculty of Data and Decision Sciences}, \orgname{Technion - Israel Institute of Technology}, \orgaddress{\city{Haifa}, \postcode{3200003}, \country{Israel}}}

\date{}

\maketitle


\abstract{We propose a novel polyhedral uncertainty set for robust optimization, 
termed the \emph{smooth uncertainty set}, 
which captures dependencies of 
uncertain parameters by 
constraining their pairwise differences. 
The bounds on these differences may be dictated by the underlying physics of the problem and may
 be expressed by domain experts. When correlations are available, the bounds can be set 
 to ensure that the associated probabilistic constraints are satisfied for any given probability. We explore specialized solution methods for the resulting optimization problems, including compact reformulations  that exploit special structures when 
 they appear, a column generation algorithm, and a reformulation of the adversarial problem as a minimum-cost flow problem. Our numerical experiments, based on problems from literature, illustrate (i) that the performance of the smooth uncertainty set model solution is similar to that of the ellipsoidal uncertainty model solution, albeit, it is computed within significantly shorter running times, and (ii) our column-generation algorithm can outperform the classical cutting plane algorithm and dualized reformulation, respectively in terms of solution time and memory consumption.
}



\section{Introduction}
\label{sec:intro}

Robust optimization has become 
accepted as an efficient framework 
for handling optimization problems whose data is plagued by uncertainty.
{Commonly used} uncertainty sets include 
polyhedral uncertainty sets (such as hyper-boxes and budgeted uncertainty sets) and ellipsoidal uncertainty sets. 
The purpose of this work is to focus on robust optimization problems for which the  {uncertain} parameters 
are spatially 
or temporally 
{bound together, thus}
having similar values.  
For instance, these parameters may describe the unknown density of some fluid at given points in space and time, the spatial and temporal propagation of which obeys the corresponding physical laws. 

\mp{Consider an optimization problem of the form}
\begin{subequations} \label{eq:prob}
\begin{align}
& \min_{(x,y)\in \feasibility} && \costx^\top x+\costy^\top y \\
& \text{subject to} && \delta^\top C^i x + d^\top_iy\leq {c}_i, && i\in[m], 
\label{eq:constr}
\end{align}
\end{subequations}
where for any $l\in\mathbb{N}$, $[l]$ denotes the index set $\{1,2,\ldots,l\}$,  
$\feasibility\subseteq
\mathcal{X}\times \mathcal Y$, 
$\mathcal{X}\subseteq\R^{n_x}_+$ and $\mathcal{Y}\subseteq \R^{n_y}_+$ 
. 
{In particular,} some components of $x$ and $y$ may be assumed to be integer. 
Also, for all $i\in [m]$, the matrix $C^i\in\R^{n\times \nx}$, the vector $d_i\in \R^\ny$, and the scalar $c_i\in\R$, as well as $\costx\in\R^\nx$ and $\costy\in \R^{\ny}$ are known parameters, and $\delta\in \R^n$ is an uncertain parameter. 
\shim{We note that $\delta$ can represent either the original uncertain parameter or its deviation from some predefined value.} 

We \mp{introduce} an uncertainty set for $\delta$ which is defined by an underlying undirected weighted  graph $G=(V,E,\gamma)$, with vertex set $V=[\nvertices]$\shim{, edge set $E$, and edge weights $\gamma$}. 
The uncertainty is defined by both 
the vertices  and 
the edges of \ng{this} graph. \ng{In particular, g}iven a nominal vector $\hat{\delta}\in \R^n$, for each vertex of the graph $i\in [n]$ the distance between \shim{$\delta_i$} and its nominal value \shim{$\hat{\delta}_i$} is bounded by parameter $\gamma_{ii}$. Thus, the uncertainty associated with the vertices is 
\ng{essentially} a box around the nominal vector $\hat{\delta}$ with size defined by $\{\gamma_{ii}\}_{i\in[n]}$. The novelty of the uncertainty set lies in 
\ng{modeling the uncertainty dependency \shim{structure} by} the edges of the graph. Specifically, for each $\{i,j\}\in E$, the magnitude of the difference between the value of \shim{$\delta_i$} and $\delta_j$ is bounded by parameter $\gamma_{ij}$. 
Accordingly, 

\begin{equation}\label{eq:defineU}
\USB\equiv\USB(\hat{\delta},G)=\set{{\delta\in \R^{\nvertices}}}{|\delta_{i}-\hat \delta_i|\leq \gamma_{ii}, \forall i\in [n], \; |\delta_i-\delta_j|\leq \gamma_{ij},\; \forall \{i,j\}\in E},
\end{equation}
where we will use $\USB$ when $\hat{\delta}$ and graph $G$ are implied 
\ng{by} the context.
This uncertainty set is referred to as a \emph{smooth uncertainty} set as it conveys a notion of \shim{closeness of}
values taken by neighboring components of the uncertain vector. 

Given the smooth uncertainty set $\USB$, we consider the robust version of the problem above, in which constraint \eqref{eq:constr} is replaced by its robust counterpart. {This leads to the robust problem
 \begin{subequations} \label{eq:robust_prob}
\begin{align}
& \min_{(x,y)\in \feasibility} && \costx^\top x+\costy^\top y \\
& \text{subject to} && 
\delta^\top C^i x + d_i^\top y\leq c_i, && i\in[m],\;\forall \delta\in\USB. 
\label{eq:robust_constr}
\end{align}
\end{subequations}
}

\mp{
Smooth uncertainty 
may naturally appear in the following applications: 
\begin{description}
\item[\bf Transportation.] Transit times along different road segments are dependent on {how congested they are.}
 {Congestion typically} propagates along upstream road segments~\citep{xiong2018predicting}, so the congestion of 
 \ngtwo{adjacent} road segments, and therefore the delays in their transit times should not vary much along the stream. \ngtwo{It can also be argued that alternative road segments (serving similar traffic demands) may not differ by much under equilibrium.}


\item[\bf Power generation with renewable resources.] A crucial aspect of the problem of switching on generating units to satisfy electric power loads over a given time horizon is the growing penetration of renewable power sources, such as wind and solar,  the production of which varies smoothly in space and time.

\item[\bf Transshipment planning.] 
In inventory management problems, retailers inventory levels must be set to minimize the costs associated with production and holding their inventories 
while facing uncertain demands and shipping excess inventories among them once the demand is observed~\citep{bertimas2022two-stage}. Evidently, retailer \ngtwo{costs as well as} demands may be highly dependent, due to their proximity in location or similarity in costumers' profiles.

\item[\bf Radiotherapy planning.] Intensity modulated radiotherapy planning (IMRT) involves setting up an array of beams with variable intensities to irradiate a tumor. Biological conditions such as oxygenation affect the radiosensitivity of the treated area~\citep{hill2015hypoxia}. Recent robust models for radiotherapy planning have accounted for temporal connections~\citep{Nohadani2017,dabadghao2025optimal} and spatial connections~\citep{goldberg2024robust} of the biological uncertainty.
\end{description}
}

In robust optimization, uncertainty sets are ideally constructed to provide probabilistic guarantees for the obtained solution. Specifically, the goal is often to design an uncertainty set such that any feasible solution to the robust problem satisfies the constraints with a prescribed probability. Many uncertainty sets can provide such guarantees when the uncertain parameters are assumed to be independent. For example, the classical ellipsoidal and box-ellipsoidal uncertainty sets~\citep{ben2000robust}, as well as the budgeted uncertainty set~\citep{bertsimas2004price}, were designed to ensure probabilistic guarantees for a linear constraint under the assumption that the uncertain parameters are independent, have a known mean, and are supported within a symmetric box around the mean (see~\cite[Section 2.3]{ben2009robust} for details). 
Recently,~\cite{bertsimas2021probabilistic} extended these results to address more general constraints and dependencies. For \shim{arbitrary} distributions, they derive both a priori and a posteriori probabilistic bounds on the satisfaction of a single constraint, provided that the robust constraint holds for a given uncertainty set. In this work, we 
\ngtwo{establish} connections between 
bounds \ngtwo{that} we derive {for} our uncertainty set and their general bounds. 

{Our proposed smooth uncertainty set is designed to explicitly capture the dependency structure of the uncertain parameters.
While correlation is often used to model dependency, such relationships may not always be fully captured by correlation coefficients. Accordingly, although we rely on the existence of correlation measures to derive improved probabilistic guarantees, explicit correlation estimates are not required to construct or apply our uncertainty set. 
Bounds on the differences between random variables may be meaningful even when computed correlations are negligible. Moreover, in many cases the data may be insufficient for reliable correlation estimation, or such estimates may not reflect differences within the parameter ranges of interest. However, underlying theoretical models (e.g., physical or economic) and expert assessments may impose constraints on the differences of the uncertain parameters. For instance, readings from nearby climate sensors may not differ beyond certain thresholds due to proximity in Euclidean distance, altitude, or other relevant factors. Thus, there are applications where our difference-bound parameters are more intuitive and applicable than linear correlation, and our model allows the direct use of expert estimates for these bounds.

An example where our uncertainty set is intuitively appealing and has already been applied is radiation therapy planning. 
Specifically, it has been used to model uncertain oxygenation levels in tumor volumes~\citep{goldberg2024robust}. Oxygenation varies spatially and temporally, and is highly uncertain; observational data may be limited, yet it must be accounted for to prescribe adequate radiation doses. This paper aims to motivate the use of this uncertainty set 
by providing probabilistic guarantees
when sufficient data is available, and
by highlighting 
its intuitive and corresponding physical characteristics in domains such as 
radiation therapy where data is lacking. 
Moreover,
we generalize and extend the computational methodology from~\cite{goldberg2024robust} to enable efficient computation in new application domains and more general optimization models. 

\shim{An additional benefit of our 
uncertainty set lies in the resulting structure of the 
\ngtwo{resulting} robust counterparts. In general,}
 the  
robust counterpart of 
discrete optimization problems {may}
introduce significant challenges. 
The difficulty arises both from a theoretical perspective, since the robust counterpart of polynomially solvable problems often becomes NP-hard (e.g., the shortest path problem~\citep{kouvelis2013robust}), and from a numerical perspective, due to the weakening of continuous relaxations (see, for example,~\cite{atamturk2006strong}). Exceptions include box uncertainty, which leads to a worst-case single-scenario reformulation,  budgeted uncertainty~~\citep{bertsimas2004price} and, more generally, polyhedral sets with a constant number of linking constraints~\citep{omer2023combinatorial}, which preserve the tractability of several important discrete optimization problems such as shortest path. 
Yet, even general polyhedral uncertainty sets tend to lead to problems that are more tractable than non-linear ones. \mp{In particular, reformulation techniques for robust (mixed-integer) linear programs  for nonlinear uncertainty set lead to (mixed-integer) nonlinear programs~(see \citep{Belotti2013} and \cite[Section 1.3]{ben2009robust}), which are  more computationally demanding to solve with commercial solvers than (MI)LPs.} \ngthree{For example this has motivated work to address the challenge arising from solving robust problems under the ellipsoidal uncertainty using polyhedral approximations of the robust counterparts~\citep{barmann2016polyhedral}.}
In this paper, we show that our smooth uncertainty set \ngthree{retains the effectiveness of modeling correlations between uncertain parameters while admitting} 
compact \ngthree{linear} reformulations 
\ngthree{of} specially structured problems arising in various applications, 
\ngthree{making their exact solution} computationally tractable.


In the current paper we propose
\ng{the novel smooth} uncertainty set  
\ng{imposing} bounds on the difference of uncertain parameters in addition to simple box bounds. \shim{In Section~\ref{sec:probbounds}, we derive probabilistic   guarantees for satisfying the robust problem's constraints,  when the correlations between the uncertain parameters are known.} In Section~\ref{sec:comp_methods}, we demonstrate the tractability of solving Problem~\eqref{eq:robust_prob}.
Specifically, in Section~\ref{sec:comapct_reform}, we study special cases, where the structure of the constraint matrix can be exploited to compactly reformulate the robust counterpart.
In Section~\ref{sec:CG}, we present a column generation algorithm for handling uncertainty polytopes that are defined by 
\ngtwo{a considerable number of} constraints, such as 
\ng{our smooth uncertainty set}, and \ng{further} tailor it to our smooth uncertainty set. 
Finally, in Section~\ref{sec:const_generation} we show that solving the adversarial problem for each \ngtwo{robust} constraint 
\ngtwo{amounts to solving} a min\ng{imum}-cost flow problem,  
which implies that it can be solved in strongly polynomial time. 
Our numerical experiments, presented in Section~\ref{sec:numerics} illustrate  the practical relevance of our new uncertainty set. 
Section~\ref{sec:cgcomp} \ngtwo{demonstrates} 
the potential efficiency of the new 
column generation algorithm, before concluding the paper with Section~\ref{sec:conclutions}.
 Additional details on the numerical experiments are given in the Appendix.

\section{Probability Bounds} 
\label{sec:probbounds}

In this section, we provide a probabilistic motivation for using the smoothed uncertainty set $\USB$. 
We assume that the uncertain parameter $\delta$ in \eqref{eq:constr} 
corresponds to a random vector $\tilde{\delta}$.
We explore 
\ngtwo{two} probabilistic settings for $\tilde{\delta}$ and show that for each setting 
the set $\USB$ \ngtwo{can be configured} so that a feasible solution to the robust constraint is guaranteed to satisfy the original constraint 
\ngtwo{involving} $\tilde{\delta}$ with 
{a given} probability.

{Throughout this section, we make the following assumption.

\begin{assumption}\label{ass:mean_variance}
{The random vector $\tilde\delta$ has mean $\mathbf{0}$, covariance matrix $\Sigma\in\R^{n\times n}$, where $\Sigma_{ij}=\cov(\tilde\delta_i,\tilde\delta_j)$, and $\Sigma_{ii}=\text{var}(\tilde\delta_i)=1$.}
\end{assumption}
\begin{remark} Assumption~\ref{ass:mean_variance} 
 is not restrictive. The case where {$\tilde\delta_i$} has expected value $\mu_i$ and variance $\sigma_i$, for each \ngtwo{$i\in[n]$}, can be addressed by applying the standardization $\tilde\delta'_i=\frac{\tilde\delta_i-\mu_i}{\sigma_i}$. The vector {$\tilde\delta'$} has expected value \ngtwo{$\mathbf{0}$} and a covariance matrix with diagonal entries equal to 1, thus satisfying Assumption~\ref{ass:mean_variance}. Note also that Assumption~\ref{ass:mean_variance} implies that matrix $\Sigma$ is  the correlation matrix of vector $\delta$.
\end{remark}

Our aim is to construct $\USB$ so that any solution satisfying the robust counterpart constraints~\eqref{eq:robust_constr} also satisfies~\eqref{eq:constr} (with $\delta$ replaced by $\tilde\delta$) with a given probability $p$. 
To this end, we observe that for any 
solution $(x,y)$ 
satisfying \eqref{eq:robust_constr}, the following inequality holds
\begin{equation}\label{eq:feas_prob_bound}
\mathbb{P}[(x,y)\text{ {satisfies} \eqref{eq:constr}} \text{ for $\tilde{\delta}$}]\geq \mathbb{P}[
{\tilde\delta}\in \USB].
\end{equation}
In Section~\ref{sec:non_param},  
we {leverage~\eqref{eq:feas_prob_bound} to} quantify the choice of parameters $\gamma$ for $\USB$ that ensure the constraints  \eqref{eq:constr} are satisfied with high probability with no additional assumption on the probability distribution.
In Section~\ref{sec:normal_dist}, we do the same, under the additional assumption that $\tilde\delta$ is normally distributed.

In both settings explored in this section, there are well-established methods of constructing an ellipsoidal uncertainty set resulting in the desired probabilistic guarantees. Specifically, these methods either depend on the multivariate Chebyshev inequality~\citep{bertsimas2021probabilistic} or, in the case of normal distribution, the Mahalanobis distance \citep{mahalanobis2018generalized}.  
 In general, one can formulate these uncertainty sets as:
\begin{equation}\label{eq:ellipsoidal_uncertaintyset} \U_{\mathcal{E}}=\{\delta\in\R^n: \delta^\top\Sigma^{-1}\delta\leq \Omega\},
\end{equation}
where $\Omega$
defines the size of the uncertainty set. To establish a connection between $\USB$ and uncertainty sets of type $\U_{\mathcal{E}}$, we show that by appropriately setting its parameters, $\USB$ can be made to enclose $\mathcal{U}_{\mathcal{E}}$,
allowing the probabilistic guarantees
that apply to
the ellipsoidal uncertainty set to carry over to $\USB$. 
{We first show a similar result for a polyhedral set defined by a general constraint matrix.} In the following, we denote
\ng{the} $jth$ row of matrix $A$ as $A_j$ (given by a row vector). 
\begin{lemma}\label{lem:sizing_polyhedral_set}
Let $A\in \R^{k\times n}$ and $\U(b)=\{\delta\in\R^n:A\delta\leq b\}$, and let $b^\E\in\R^n$. If  
$b_i^{{\mathcal{E}}}=\sqrt{\Omega}\norm{A_i\Sigma^{1/2}}$ for all $i\in[k]$, then 
$\U_\E\subseteq \U(b^{{\mathcal{E}}})$ and therefore $\mathbb{P}({\tilde\delta}\in \U_\E)\leq \mathbb{P}({\tilde\delta}\in \U(b^{{\mathcal{E}}}))$.
\end{lemma}
\begin{proof}
To show $\U_\E\subseteq\U(b^{{\mathcal{E}}})$, we define the auxiliary vector ${\delta'}=\Sigma^{-1/2} \delta$ and use it to obtain a lower bound on $b^{{\mathcal{E}}}$. 
Specifically, observe that for all $i\in [k]$
$$\max_{\delta\in \U_{\mathcal{E}}}{A_i}\delta=\max_{\norm{\delta'}\leq 1} \sqrt{\Omega}({A_i} \Sigma^{1/2}{\delta'})=\sqrt{\Omega}\norm{{A_i}\Sigma^{1/2}}=b_i^{\mathcal{E}}.$$
Thus, any $\delta\in \U_\E$ satisfies $\delta\in \U(b^{{\mathcal{E}}})$, and 
{so} $\U_\E\subseteq\U(b^{{\mathcal{E}}})$\ngtwo{, from which it  immediately follows that $\mathbb{P}(\tilde\delta\in \U_{\mathcal{E}})\leq \mathbb{P}(\tilde\delta\in \U(b^{{\mathcal{E}}}))$}.
\end{proof}
In order to apply Lemma~\ref{lem:sizing_polyhedral_set} to our uncertainty set $\USB$, we introduce the following notation 
\ngtwo{for} the coefficient vectors of the constraints in $\USB$:   
\ngtwo{for $i,j\in[n]$,} $e_{ij}=
{e_i-e_j}$ 
if $i\neq j$, 
and otherwise
$e_{ij}=e_i$, 
where $e_i$ is the standard unit vector in $\R^n$. Using the structure of the constraints in $\USB$ we directly obtain the following result.

\begin{corollary}\label{cor:ellisoid_construction}
\shim{Given $\U_{\E}$,} let $\USB^{\E}$ be equal to $\USB$ with $\hat{\delta}=\mathbf{0}$ and $\gamma_{ij} =
\gamma_{ij}^{\E}$, where  $\gamma_{ij}^{\E}\geq \sqrt{\Omega}\norm{\Sigma^{1/2}e_{ij}},$
for all $i,j\in [n]$ such that $\{i,j\}\in E$ or $i=j$. Then,  $\U_{\mathcal{E}}\subseteq\USB^{\mathcal{E}}$, and thus
$\mathbb{P}({\tilde\delta}\in\U_{\mathcal{E}})\leq \mathbb{P}({\tilde\delta}\in\USB^{\mathcal{E}}).$
\end{corollary}

\shim{\ngtwo{Probabilistic} guarantees for other 
polyhedral sets can 
be obtained 
\ngtwo{following a similar analysis. For example,} by a simple change of variables the rotated box of radius $r$ given by $\mathcal{U}_{RB}=\{\delta\in\R^n: \norm{\Sigma^{-1/2}\delta}_\infty\leq r\}$
(using standardization~\cite[Remark 1.2.2]{ben2009robust})},
\ngtwo{enjoys a probabilistic guarantee that is established in the following corollary.}
\begin{corollary}\label{cor:ellisoid_rotated_box}
Let $\mathcal{U}_{RB}^{\E}$ be a rotated box with radius $r=r^{\E}$ where $r^\E\geq \sqrt{\Omega}$. Then,  $\U_{\mathcal{E}}\subseteq\U_{RB}^{\E}$, and thus
$\mathbb{P}({\tilde\delta}\in\U_{\mathcal{E}})\leq \mathbb{P}({\tilde\delta}\in\U^{\E}_{RB}).$
\end{corollary}}

\ngtwo{Note that} {while  Corollary~\ref{cor:ellisoid_construction} 
\ngtwo{relates} ellipsoidal and smooth uncertainty sets, the construction of $\USB$ as an enclosing set of $\U_{\mathcal{E}}$ may 
be overly conservative \ngtwo{in some cases}. In the following sections, we also explore alternative constructions of $\USB$ and compare the resulting sets.}

\subsection{Constructing  $\USB$ for General Random Variables}\label{sec:non_param}
The analysis in this section extends general probability bounds 
from the literature 
to motivate our uncertainty set for  modelling  correlated random variables.
The following probability bound for the ellipsoidal uncertainty set applies to general random variables with a positive definite covariance matrix (or more generally using a positive definite matrix that serves as a bound on the covariance matrix).

\begin{lemma}[Multivariate Chebyshev {Inequality}]
\label{lem:chebyshev}
Assume
that $\Sigma\succ 0$, then 
$\mathbb{P}[{\tilde\delta}\notin \U_{\mathcal E}] \leq n/\Omega
$ \end{lemma} For a proof see, for example,~\cite{navarro2016very}. Next, the following establishes a probability bound, formulating the construction of $\USB$ for a given probability, combining the results of and Corollary~\ref{cor:ellisoid_construction} and Lemma~\ref{lem:chebyshev}. 

\begin{proposition}\label{prop:genprob}
Given a probability $p\in (0,1)$, letting $\gamma^{{\E}}_{ij}\geq \sqrt{n/p}\vectornorm{\Sigma^{1/2}e_{ij}}$
 for all $i,j$ such that $\{i,j\}\in E$ or $i=j$, 
guarantees that $\mathbb{P}[{\tilde\delta}\in\USB^{{\E}}]\geq 1-p$.  
\end{proposition}

The following five-dimensional example will serve both as an illustration and as a basis for comparing the uncertainty sets constructed under the different probabilistic assumptions and bounds developed in this and the subsequent sections.
\begin{example}\label{example:expdeccor}
Consider $\tilde{\delta}\in\R^n$ with a 
{covariance matrix $\Sigma\in\R^{n\times n}$,  given by $\Sigma_{ij}=\hat\rho^{\abs{j-i}}$} 
for all $i,j\in[n]$. 
{Note that
for any $\hat\rho\in [0,1)$, matrix $\Sigma$ {is of full rank} and hence invertible}. 
{$\USB^{{\E}}$ is constructed based on a complete graph $G$ 
with} 
$\gamma^{{\E}}$ defined as in 
Proposition~\ref{prop:genprob}. 
The volume and diameter are compared in Table~\ref{tbl:uncertainty_set_comparison} for $n=5$ and $p=0.01$ with 
{uncertainty sets} derived 
from 
probability bounds developed for the normal probability distribution in the next section.
\shim{We note that we are able to compute the diameter and volume of the sets due to their small dimension ($n=5$), although computing these 
measures for general convex sets is known to be hard~\cite{Khachiyan1993}.}
\end{example}


While the construction of $\USB$ in Proposition~\ref{prop:genprob} is based on a lower bounds on $\mathbb{P}[\tilde\delta\in\USB]$ and hence a lower bound on the probability of all constraints being satisfied, the analysis that follows {develops} bounds on the probability of a single constraint being violated (or satisfied). 
{\ngtwo{For convenience, in the following} let the feasible set with respect to the robust constraint being imposed with  uncertainty set $\mathcal U$ be denoted by $\Z(\mathcal{U})=\{(x,y)\in \Z: \max_{\delta\in \U}\delta^\top Cx+d^\top y\leq c\}$.} 
\begin{lemma}[{\cite[Proposition 5.6]{bertsimas2021probabilistic}}]\label{lem:bersimas}
{Let $\U$ be a given uncertainty set}. {For all $(x,y)\in\Z(\mathcal U) $, 
$\mathbb{P}[
\tilde\delta^\top Cx + d^\top y > c]
\leq [1+R^2
/\lambda_{\max}(\Sigma)]^{-1}$,} 
where $R$ is the radius of the largest ball enclosed in $\U$. 
\shim{Thus, for the ellipsoidal uncertainty set $\U_{\E}$, $R$ is given by $R=\sqrt{\lambda_{\min}(\Sigma)\Omega}$.}
Further, for a polyhedral set $\U=\set{\delta\in\R^n}{A\delta\leq b}$, where $A\in\R^{k\times n}$ \shim{and $b\in\R^k_{++}$}, $R$ is given by $R=\min_{i\in[k]}b_i/\vectornorm{A_i}_2$. 
\end{lemma} 

\noindent Specializing this result for our smooth uncertainty set, $\USB$, results in 
\newline $R=\min\{\frac{1}{\sqrt{2}}\min_{\{i,j\}\in E}\gamma_{ij},\min_{i\in[n]}\gamma_{ii}\}^2/\lambda_{\max}(\Sigma)]^{-1}$.   
{Now, the following result extends 
the one from \cite{bertsimas2021probabilistic}, \shim{by replacing $R$ with the largest $\Omega$ for which $\U_\E$ is contained in the uncertainty set.} 

\begin{lemma}
\label{lem:bertsimasext}~~
\begin{enumerate}[label=(\roman*)]
\item\label{lem:bersimasnormext} For all $(x,y)\in\Z(\mathcal U_{{\E}})$,  $\mathbb{P}[\tilde\delta^\top C x + d^\top y > c]\leq [1+\Omega]^{-1}$.
\item\label{lem:bersimaspolyext} For $\U=\set{\delta\in\R^n}{A\delta\leq b}$ and all $(x,y)\in\Z(\mathcal U)$,   
$\mathbb{P}[\tilde\delta^\top 
C x + d^\top y > c]\leq [1+R^2]^{-1}$, where $R=\min_{i\in[k]}b_i/{\vectornorm{A_i \Sigma^{1/2}}}$. 
In particular, if  $\U=\USB$ then $R=\min\limits_{i,j\in [n]:\{i,j\}\in E\vee (i=j)}\gamma_{ij}/\vectornorm{\Sigma^{1/2}e_{ij}}$, \shim{if $\U=\U_{RB}$ then $R=r$.}
\end{enumerate}
\end{lemma}
\begin{proof}
Define the random variable ${\tilde\delta'}=
\Sigma^{{-}1/2}{\tilde\delta}$. Evidently this random variable has a 
covariance matrix given by $\tilde\Sigma=I$. {\ngtwo{Accordingly}, we do a change of variables using the same transformation with $\delta'=\Sigma^{-1/2}\delta$ for our uncertainty sets, with $\delta'$ representing the random vector $\tilde\delta'$.}  

\noindent\textbf{Proof  of~\ref{lem:bersimasnormext}.} 
{Substituting $\delta=\Sigma^{1/2}\delta'$ 
in the definition of $\U_{\E}$, 
the equivalent uncertainty set with respect to $\delta'$ is a ball of radius $\sqrt{\Omega}$.} 
Since 
{{$(x,y)\in \Z(\mathcal U_{{\E}})$}  if and only if $\max_{{\delta'\in\R^n}:\norm{{\delta'}}\leq \sqrt{\Omega}}{{(\delta')}^\top \Sigma^{1/2}} C x + d^\top y \leq c$,}
by Lemma~\ref{lem:bersimas}, {$(x,y)$} {satisfies}
\[\mathbb{P}[\tilde\delta^\top C x + d^\top y > c]=\mathbb{P}[{(\tilde\delta')}^\top {{\Sigma^{1/2}}}C x + d^\top y > c]\leq [1+\Omega/\lambda_{\max}(\tilde \Sigma)]^{-1} = [1+\Omega
]^{-1}.
\]
\textbf{Proof of~\ref{lem:bersimaspolyext}.} 
Let $\U' = \set{\delta'\in\R^n}{A\Sigma^{1/2}{\delta'}\leq b}$. Then, by Lemma~\ref{lem:bersimas}, the {largest} ball inscribed in ${\U'}$ has radius $\min_{i\in[k]}\frac{b_i}{\vectornorm{A_i \Sigma^{1/2}}}$. 
Since  
\[
c\geq \max_{\delta\in\U}\delta^\top C x + d^\top y = \max_{{\delta'\in \U'}}{(\delta')}^\top \Sigma^{1/2}Cx+d^\top y, 
\] 
and thus $
\mathbb{P}[\tilde\delta^\top C x + d^\top y > c]=\mathbb{P}[{(\tilde\delta')}^\top \Sigma^{1/2}C x + d^\top y > c]
\leq \left[1+\min_{i\in[k]}\frac{b_i^2}{\vectornorm{A_i \Sigma^{1/2}}^2}\right]^{-1}$.
\end{proof}

{Note that Lemma~\ref{lem:bertsimasext} strictly improves upon the a priori 
bound of Lemma~\ref{lem:bersimas} when uncertainty sets and/or their robust counterparts can be constructed based on (known) $\Sigma$, and $\lambda_{\min}(\Sigma)<\lambda_{\max}(\Sigma)$.} 
Based on the prior lemmas, the following proposition establishes a construction of $\USB$ that guarantees a single-constraint violation probability of at most $p$.}
\begin{proposition}\label{prop:singleconsviol}
Given probability $p\in (0,1)$, letting $\hat{\delta}=\mathbf{0}$ and $
\gamma_{ij}\geq \sqrt{1/p-1}\vectornorm{\Sigma^{1/2}e_{ij}}
$ for all $i,j\in[n]$ such that $\{i,j\}\in E$ or $i=j$, guarantees for all {$(x,y)\in\Z(\USB)$}  that $\mathbb{P}[\tilde\delta^\top C x + d^\top y > c]\leq p$.
\end{proposition}
{The proof of this proposition follows directly from Lemma~\ref{lem:bertsimasext}-\ref{lem:bersimaspolyext} or Lemma~\ref{lem:bertsimasext}-\ref{lem:bersimasnormext} combined with  Corollary~\ref{cor:ellisoid_construction}.} {Note that the result of Proposition~\ref{prop:singleconsviol} implies that the uncertainty set can be constructed to satisfy a violation probability of at most $p$ using $\gamma_{ij}$'s that are a factor ${\sqrt{(1-p)/n}}$ smaller than $\gamma^{\E}_{ij}$ implied by Proposition~\ref{prop:genprob}. However, also note that Proposition~\ref{prop:singleconsviol} establishes a probability bound $p$ on the event that only a single constraint is violated. \shim{We} bound the probability of the event that any of $m$ constraints is violated by at most $p$ by the union bound $\gamma_{ij}\geq \sqrt{m/p-1}\vectornorm{\Sigma^{1/2}e_{ij}}$. Hence, this bound is preferable as long as $m\leq n+p$, meaning $m\leq n$ (since $p\in(0,1)$).



\subsection{Constructing $\USB$ for \shim{Normally Distributed Uncertainty}}\label{sec:normal_dist}
In this section, we treat the case where $\tilde\delta$ is a random vector  
{that is $n$-variate normally distributed.} To this end, we augment Assumption~\ref{ass:mean_variance} by the following assumption that applies throughout the current
section.
\begin{assumption}\label{ass:normal_delta}
    The random vector $\tilde\delta$ 
 is normally distributed, 
 and $\Sigma$ is invertible.
\end{assumption}
Note that under Assumption~\ref{ass:normal_delta} the random variable $\tilde\delta^\top\Sigma^{-1}\tilde\delta$ 
follows a $\chi^2(n)$ distribution. Therefore, the ellipsoidal uncertainty set $\U_\E$ with $\Omega=\chi^2_{1-p}(n)$, 
where $\chi^2_{1-p}(n)$ is the $1-p$ quantile of the  $\chi^2$ distribution with $n$ degrees of freedom, satisfies $\mathbb{P}(\tilde\delta\in\U_{\mathcal{E}})\geq 1-p.$ 
In the following, we develop two alternatives for 
determining the values of $\gamma_{ij}$,
to ensure that  $\delta\in \USB$ with probability of at least $1-p$. 
The next result follows directly from {the construction of $\USB$ to enclose the ellipsoid with radius $\sqrt\Omega$ from} Corollary~\ref{cor:ellisoid_construction}, combined with the setting of $
\Omega=\chi^2_{1-p}(n)$.  
\begin{proposition}
\label{prop:ellisoid_construction_normal}
Given $p\in(0,1)$
letting $\gamma_{ij}^{\E}=\sqrt{\chi^2_{1-p}(n)}\norm{\Sigma^{1/2}e_{ij}},$
for all $i,j\in [n]$ such that $\{i,j\}\in E$ or $i=j$, 
where $\chi^2_{1-p}(n)$ is the $1-p$ quantile of the  $\chi^2$ distribution with $n$ degrees of freedom,  
guarantees that $\mathbb{P}[\delta\in\USB^{{\E}}]\geq 1-p$.
\end{proposition}
Alternatively, the following {proposition} shows {how to determine the required} values 
{of} $\gamma_{ij}$ based on the standard normal distribution.
\begin{proposition}\label{prop:union_bound_construction_normal}
Given $p\in (0,1)$, 
let  
$\gamma_{ij}^{Z}=Z_{1-\alpha/2}\norm{\Sigma^{1/2}e_{ij}}$
for all $i,j\in [n]$ such that $\{i,j\}\in E$ or $i=j$, where $Z_{1-\alpha/2}$ is the $1-\alpha/2$ quantile of the standard normal distribution and
$\alpha=\frac{p}{|E|+n}$. Then, $\USB^Z$ defined as $\USB$ with $\hat{\delta}=\mathbf{0}$ and $\gamma_{ij}=\gamma_{ij}^{Z}$ guarantees
$\mathbb{P}(\delta\in\USB^Z)\geq 1-p.$
\end{proposition}
\begin{proof}
From Assumption~\ref{ass:normal_delta}, for all {$\{i,j\}\in E$} the scalar $\tilde\delta_i-\tilde\delta_j=e_{ij}^\top\Sigma^{1/2}\tilde\delta'$ has a normal distribution with mean zero and variance $e^\top_{ij}\Sigma e_{ij}$. Therefore, it follows from the definition of the normal distribution that for all 
$\{i,j\}\in E$, 
\begin{align*}
\mathbb{P}({|\tilde\delta_i-\tilde\delta_j|}\geq \gamma_{ij}^Z) = \mathbb{P}\left(\frac{|\tilde\delta_i-\tilde\delta_j|}{\sqrt{e_{ij}^\top\Sigma e_{ij}}}\geq Z_{1-\alpha/2}\right)= \alpha,\;
\mathbb{P}({|\tilde\delta_i|}\geq \gamma_{ii}^Z)= \mathbb{P}\left(\frac{|\tilde\delta_i|}{\sqrt{e_{i}^\top\Sigma e_{i}}}\geq Z_{1-\alpha/2}\right)= \alpha.
\end{align*}
Using the union bound, we therefore obtain that
$$\mathbb{P}(\tilde\delta\in\USB^Z)=1-\mathbb{P}(\tilde\delta\notin\USB^Z)\geq 1-\sum_{\{i,j\}\in E} \mathbb{P}(|\tilde\delta_i-\tilde\delta_j|\geq \gamma_{ij}^Z)-\sum_{i\in [n]} \mathbb{P}(|\tilde\delta_i-\mu|\geq \gamma_{ii}
)=1-p.$$
\end{proof}

Proposition~\ref{prop:ellisoid_construction_normal} and Proposition~\ref{prop:union_bound_construction_normal}, combined with \eqref{eq:feas_prob_bound}, directly imply the following {corollary} 
defin{ing} a {minimal} construction of $\USB$, while 
guaranteeing a {bound on the} 
probability of {violating} the constraints.

\begin{corollary}
\label{cor:gengaussprobbound}
Let $p\in (0,1)$, and $\beta=\min\left\{\sqrt{\chi^2_{1-p}(n)},Z_{1-p/(2(|E|+n))}\right\}$. Then, setting $\hat{\delta}=\mathbf{0}$ and 
$\gamma_{ij}\geq \beta\sqrt{e_{ij}^\top\Sigma e_{ij}}$ for all $i,j\in [n]$ in $\USB$, guarantees that any pair $(x,y)\in\mathcal{Z}(\USB)$ 
violates constraints \eqref{eq:constr} with probability of at most $p$.
\end{corollary}

\begin{remark}
    When $\Sigma$ is not invertible, it follows from the {fact that $\Sigma\succeq 0$} that there exists a matrix $R\in\R^{n\times q}$,  for some $q<n$, such that $R$ has full column {rank} and $\Sigma=R R^\top$. Moreover, there exists a  random vector $\tilde{\xi}$  normally distributed with expected value $\mathbf 0\in\R^q$ and covariance matrix $I\in\R^{q\times q}$, meaning {that} its components are normally distributed independent random variables, such that $\tilde\delta=R\tilde{\xi}$. Thus, we can alternatively define the (degenerate) ellipsoidal uncertainty set as
    $\U_{\mathcal{E}}=\{\delta\in\R^n: \exists \xi\in\R^q, \delta=R \xi, \norm{\xi}^2\leq \chi_{1-p}(q)\}.$
    Given these definitions, we can adjust Proposition~\ref{prop:ellisoid_construction_normal} and Proposition~\ref{prop:union_bound_construction_normal} to obtain $\gamma_{ij}^{\mathcal{E}}=\sqrt{\chi^2_{1-p}(q)}\norm{R^\top e_{ij}}$ and $\gamma_{ij}^{Z}={Z_{1-p/(2(|E|+n))}}\norm{R^\top e_{ij}}$, respectively.
\end{remark}
\shim{We note that since for this case the elements of $\tilde{\delta}'=\Sigma^{-1/2}\tilde{\delta}$ are independent random variables, similarly to Propositions~\ref{prop:ellisoid_construction_normal} and \ref{prop:union_bound_construction_normal}, we can also construct a rotated-box  
uncertainty set that provides a probabilistic guarantee os follows.}
\shim{
\begin{corollary}
\label{cor:gengaussprobboundRB}
Let $p\in (0,1)$, and $\beta=\min\left\{\sqrt{\chi^2_{1-p}(n)},Z_{1-p/(2n)}\right\}$. Then, setting 
$r\geq \beta$ in $\U_{RB}$, guarantees that any pair $(x,y)\in\mathcal{Z}(\U_{RB})$ 
violates constraints \eqref{eq:constr} with a probability of at most $p$.
\end{corollary}}

\begin{table}[t]
\centering
\vspace{6pt}
\begin{tabular}{ll|rr|rr}
\hline
\textbf{Distribution} & \textbf{Uncertainty Set} & \multicolumn{2}{|c|}{\textbf{Volume}} & \multicolumn{2}{|c}{\textbf{Diameter}}\\
\hline
& & $\rho=$0.2& $\rho=$0.8 & $\rho=$0.2& $\rho=$0.8\\
\hline
\multirow{4}{*}{General} & $\U_{\E}$&  2.71e+07 (1.00) & 3.81e+06 (1.00)   
 &  52.89 & 85.27 \\
& $\USB^{\E}$ 
& 8.27e+07 (3.05) & 1.25e+07 (3.28) & 100.00 & 100.00\\
& Enclosing box & 1.79e+08 (6.60) & 1.79e+08 (46.91) &  100.00 & 100.00\\
& $\U_{RB}$ &
1.65e+08 (6.08) &2.32e+07 (6.08) &     117.26
 &190.30\\
\hline
\multirow{5}{*}{Normal} & $\U_{\E}$&  4.29e+03 (1.00)
 &  603.06 (1.00)
 & 9.19 & 14.81\\
& $\USB^{\E}$ & 1.31e+04 (3.05)
& 1.98e+03 (3.28) & 17.37 & 17.37\\
& $\USB^{Z}$ & 6.75e+03 (1.58)
& 1.02e+03  (1.69) & 15.22
& 15.22\\
& Enclosing box & 2.83e+04 (6.60) & 2.83e+04 (46.91) & 17.37& 17.37\\
& $\U_{RB}$ &
8.31e+03 (1.94) &1.17e+03 (1.94) &    16.21
 &26.30\\
\hline
\end{tabular}
\caption{Comparison of uncertainty set volume for $n=5$ and $p=0.01$, with the ratio with volume of $\U_{\mathcal{E}}$ in parentheses, and diameter for various uncertainty set ensuring violation probability of no more than $p$ for Example~\ref{example:expdeccor}.\label{tbl:uncertainty_set_comparison}}
\end{table}


\begin{figure}[t] 
\centering
\subfloat[$\rho=0.2$\label{fig:rho0.2}]{\includegraphics[width=0.5\linewidth]{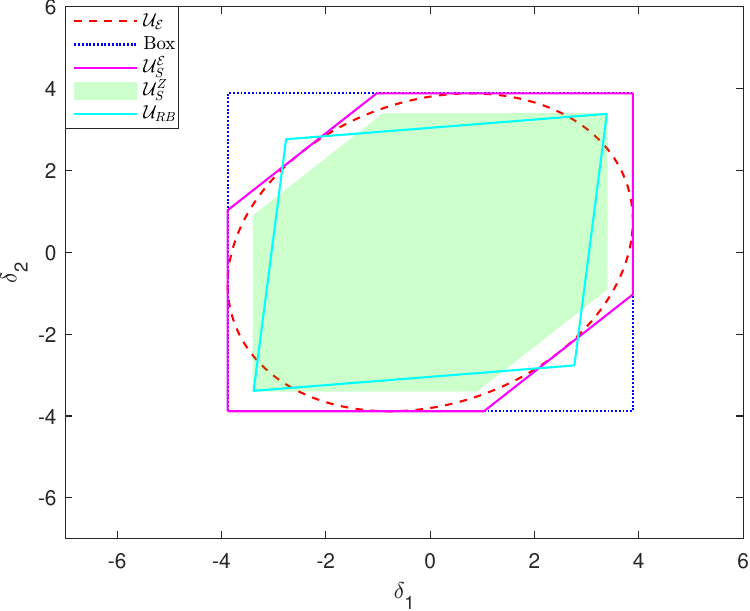}}
\subfloat[$\rho=0.8$\label{fig:rho0.8}]{\includegraphics[width=0.5\linewidth]{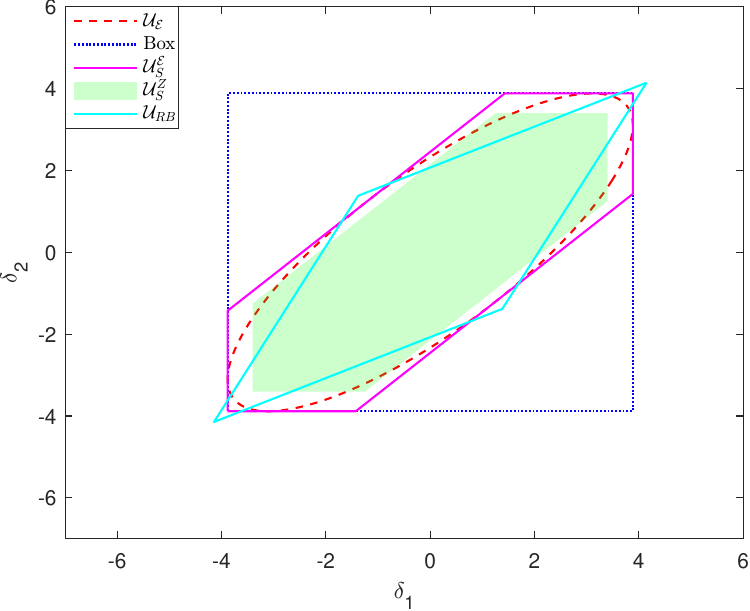}}
\caption{Different uncertainty sets for  Example~\ref{example:normal} with $n=5$ and $p=0.01$.\label{fig:uncertainty_set_comparison}}
\end{figure}

\begin{example}[Example 1 revisited]\label{example:normal}
In the following simple example, which illustrates Corollary~\ref{cor:gengaussprobbound}, we consider the same setting as in Example~\ref{example:expdeccor}, with the added assumption that $\tilde\delta$ is normally distributed. 
We build the uncertainty set $\USB$ 
according to either Proposition~\ref{prop:ellisoid_construction_normal} ($\USB^{\mathcal{E}}$) or Proposition~\ref{prop:union_bound_construction_normal} ($\USB^Z$) to guarantee a violation probability of at most $p$. \shim{The sizes of the resulting uncertainty sets, as 
indicated by the volume and diameter
, are summarized in  Table~\ref{tbl:uncertainty_set_comparison}.} 
Figure~\ref{fig:uncertainty_set_comparison} shows the same uncertainty sets projected on the plane corresponding to $\delta_1$ and $\delta_2$.  We can see that in this case the uncertainty set $\USB^Z$ is smaller than $\USB^{\mathcal{E}}$, and that $\USB^Z$ does not contain and is not contained in $\U_{\mathcal{E}}$. Moreover, while $\USB^Z$ is about 1.6-1.7 times larger than $\U_{\mathcal{E}}$ in terms of volume, it is significantly smaller than the box, and slightly smaller than the rotated-box uncertainty set. 
Thus, we can conclude that $\USB^Z$ is a reasonable polyhedral alternative for $\U_{\mathcal{E}}$ in this case.
\end{example}

\begin{remark}
A numerical evaluation comparing $Z_{1-\alpha/2}$ and $\chi^2_{1-p}$ for a complete graph suggests that the minimal $p$ for which $\chi^2_{1-p}<Z_{1-\alpha/2}$ increases with $n$ and is equal to $0.26$ when $n=3$. Thus, we can conclude that for small values of $p$ the smooth uncertainty set defined in Proposition~\ref{prop:union_bound_construction_normal} will be smaller than the smooth uncertainty set defined in  Proposition~\ref{prop:ellisoid_construction_normal}.
\end{remark}

\section{Compact Reformulations and Computational Methods}
\label{sec:comp_methods}

A straightforward reformulation of \eqref{eq:prob} using dualization is given by the following optimization problem
 \begin{subequations} \label{eq:robust_reformulation}
\begin{align}
& \min_{(x,y)\in \feasibility\atop\alpha\in\R^{(\nvertices+\nedges)m}} && \costx^\top x+\costy^\top y \\
& \text{subject to} && 
\alpha^i_{jj}
+ \sum_{\substack{k\in[n]:k<j,\;\\
\{k,j\}\in E}} \alpha^i_{kj}- \sum_{\substack{k\in[n]:k>j,\;\\\{k,j\}\in E}} \alpha^i_{jk}=C^i_j x, && i\in[m],\; j\in[n]\\
&&& \sum_{j\in[n]} (\gamma_{jj}|\alpha^i_{jj}|+\hat{\delta}_j\alpha^i_{jj})
+
\sum_{\substack{k, j\in[n]:k<j,\\
\{k,j\}\in E}} \gamma_{kj}|\alpha^i_{kj}|  + d_i^\top y\leq c_i, && i\in[m].
\end{align}
\end{subequations}
Comparing this formulation with the nominal formulation in \eqref{eq:prob}, it requires an additional $(|E|+n)m$ variables, 
the $m$ linear constraints are replaced by $m$ nonlinear constraints, and $mn$ linear constraints are added. Alternatively, an equivalent linear reformulation of this problem, results in $2(|E|+n)m$
additional 
 variables and $mn$  additional linear constraints with respect to the nominal problem \eqref{eq:prob}.
Therefore, it is evident that the size of this naive reformulation becomes restrictive when $n$ and $|E|$ are large. However, in this \shim{section} we show how the properties of the constraint matrix and sparsity of the graph can be used either to obtain a more compact reformulation or to construct efficient methods to solve the robust problem.


\subsection{Exact Compact Reformulations}
\label{sec:comapct_reform}
In this section, we discuss special structures of constraint \eqref{eq:constr} for which the robust constraint \eqref{eq:robust_constr} can be replaced by a compact reformulation simpler than~\eqref{eq:robust_reformulation}.

For each $i\in[m]$, we consider the sets, 
\begin{align*}  S^+_i\equiv\set{j\in [n]}{C^i_j x\geq 0\;\; \forall (x,y)\in\mathcal Z} && \text{ and } &&  S^-_i\equiv\set{j\in [n]}{C^i_j x\leq 0\;\; \forall (x,y)\in\mathcal Z}.\end{align*}
In particular, for $i\in [m]$, $S^{++}_i\equiv\set{j\in S^+_i}{C^i_j x > 0 \;\; \text{ for some } (x,y)\in\mathcal Z}$ $=S_i^+\setminus S^-_i $ and similarly $S^{--}_i= S^-_i\setminus S_i^{+}$. 
For simplicity, we look at one constraint $i$ and omit the index $i$ from the notation of the $i$th constraint in \eqref{eq:constr}. Specifically, we look at a general constraint of the form
\begin{equation}\label{eq:gen_cons}\delta^\top C x+ d^\top y\leq c,\quad \forall {\delta}\in \USB\end{equation}
with sets $S^+$ and $S^-$ defining the components of $\delta$ which have nonnegative 
and nonpositive 
coefficients, respectively, for all $(x,y)\in\mathcal Z$,
and $S^{++}\subseteq S^+$ and $S^{--}\subseteq S^-$ denote those coefficients that are also nonzero for some $(x,y)\in \mathcal Z$.
We analyze two cases: (1) either $S^+$ or $S^-$ is equal to $[n]$, and (2) there exists a $j\in[n]$ such that $S^{++}=\{j\}$ and $S^-=[n]\setminus\{j\}$. These cases are of interest since they appear in applications such as the shortest path, equitable food distribution, and radiotherapy treatment planning mentioned in Section~\ref{sec:intro}, among many others.

Before 
presenting our results 
for these two cases, we start with 
some 
preliminary definitions and lemmas. 
Recall that $\USB$ is defined by a graph $G=(V,E,\gamma)$ where $V=[n]$. 
For  
a pair of nodes $i,j\in [n]$, such that $i\neq j$,  
$\dist(i,j)$ is defined as the shortest distance between vertex $i$ and vertex $j$ on graph $G$, and $\dist(j,j)=0$ for all $j\in [n]$. Of course, this distance function $\dist: V\times V\rightarrow \R_+$  is a metric. We can now 
consider the metric closure of our
graph, given by $\tilde{G}=(V,V\times V,\dist)$, and obtain the following equivalence between the uncertainty sets defined by graphs $G$ and $\tilde{G}$ directly from the definition of $\dist$. 
\begin{lemma}\label{lem:equiv_uncertain_sets}
Let $\hat{\delta}\in\R^n$. Then $\USB\equiv\USB(\hat{\delta},G)=\USB(\hat{\delta},\tilde G)$.
\end{lemma}
Next we define new bounds on the values of $\delta_j$ in the set $\USB$. For each $j\in [n]$, 
initial upper and lower bound 
$\ubar{\delta}^0_j=
\hat{\delta}_j-\gamma_{jj}$ 
and $\bar{\delta}^0_j=\hat{\delta}_j+\gamma_{jj}$ are directly implied by parameter $\gamma_{jj}$. 
A  
tightening of these bounds is obtained once considering the connectivity of graph $\tilde{G}$ 
for each vertex $j\in [n]$ with respect to all other vertices,
\begin{equation}\label{eq:def_bar}
\ubar{\delta}_j=
{\max_{k\in V} \{\ubar{\delta}^0_k-\dist(k,j)\}},
\quad \bar{\delta}_j=\min_{k\in V} \{\bar{\delta}^0_k+\dist(k,j)\}.
\end{equation}
The following lemma establishes that the interval created by these bounds is, in fact, a one dimensional projection of our uncertainty set onto its $j$th component. The proof of this lemma follows directly from~\cite[Proposition 1]{goldberg2024robust}  applied to $\USB(\hat{\delta},\tilde G)$.
\begin{proposition}\label{prop:eq_bar}
For each $j\in [n]$, the segment
$[\ubar{\delta}_j,\bar{\delta}_j]$ with bounds defined 
{by} \eqref{eq:def_bar} is the projection of $\USB$ onto the $j$th coordinate.
\end{proposition} 
A direct result of Lemma~\ref{lem:equiv_uncertain_sets} and Proposition~\ref{prop:eq_bar}
is that 
\begin{equation}\label{eq:USB_for_3}\USB= \left\{ \delta\in \R^n:\begin{array}{cc}
         |\delta_l-\delta_k| \leq \dist(l,k) & \{l,k\}\in V\times V:l\neq k \\
         \ubar{\delta}_l\leq \delta_l\leq \bar{\delta}_l& l\in V\\
\end{array}\right\}.
\end{equation}
Moreover, using this 
alternative definition of $\USB$ (possibly including many redundant constraints), we prove the following lemma
, which outlines some useful properties of $\USB$.
\begin{lemma}\label{lem:two_members_ubs}~
The vectors $\ubar{\delta},\bar{\delta}$, with components defined by \eqref{eq:def_bar} belong to $\USB$.
\end{lemma}
\begin{proof}
For every $j\in [n]$, let $\bar{k}_{j}\in \arg\max_{k\in V} \{\ubar{\delta}^0_k-\dist(k,j)\}$ and let  
    $\ubar{k}_j\in \arg\min_{k\in V} \{\bar{\delta}^0_k+\dist(k,j)\}$.
    Then, for any $l,j\in[n]$, 
\begin{align*}
    \ubar{\delta}_j-\ubar{\delta}_l&\leq \ubar{\delta}^0_{\bar{k}_j}-\dist(\bar{k}_j,j)-\ubar{\delta}^0_{\bar{k}_j}+\dist(\bar{k}_j,l)=\dist(\bar{k}_j,l)-\dist(\bar{k}_j,j)\leq \dist(l,j),\\
    \bar{\delta}_j-\bar{\delta}_l&\leq \bar{\delta}^0_{\ubar{k}_l}+\dist(\ubar{k}_l,j)-\ubar{\delta}^0_{\ubar{k}_l}-\dist(\ubar{k}_l,l)=\dist(\ubar{k}_l,j)-\dist(\ubar{k}_l,l)\leq \dist(l,j),
    \end{align*}
    where the first inequality in both lines follows from \eqref{eq:def_bar} and the last inequality in both lines follows from the triangle inequality. 
    Thus, equation \eqref{eq:USB_for_3} implies that $\ubar{\delta}$ and $\bar{\delta}$ are in $\USB$. 
\end{proof}

Proposition~\ref{prop:eq_bar} allows us to compute the projection of $\USB$, that is the bounds $\ubar{\delta}$ and $\bar{\delta}$, via preprocessing.
These vectors prove useful for deriving a compact reformulation for the first case, given by the following result.
\begin{proposition}\label{prop:all_same_sign}
Let $\ubar{\delta}$ and $\bar{\delta}$ be defined by \eqref{eq:def_bar}.
If $S^-=[n]$, then \eqref{eq:gen_cons} is equivalent to 
$\ubar{\delta}^\top C x +d^\top y\leq c.$
Similarly, if $S^+=[n]$, then
\eqref{eq:gen_cons} is equivalent to 
$\bar{\delta}^\top C x +d^\top y\leq c.$
\end{proposition}
\begin{proof}
Without loss of generality consider the case where $S^-=[n]$ (the case where $S^+=[n]$ is analogous using the opposing signs for the coefficients of $\delta$ and replacing $\ubar\delta$ everywhere with $\bar\delta$). It is straightforward that
$
    \max_{\delta\in \USB} \delta^\top C x\leq \max_{\delta\in[\ubar{\delta},\bar{\delta}]} \delta^\top C x = \ubar{\delta}^\top C x,
$ 
where the first inequality follows from the definition of $[\ubar\delta_j,\bar\delta_j]$ as the projection of $\USB$ onto the $j$th component, and the equality follows from $j\in [n]=S^{-}$ implying that 
$C_j x\leq 0$ for all $j\in [n]$ and $(x,y)\in \mathcal{Z}$. By Lemma~\ref{lem:two_members_ubs}, 
$\ubar\delta\in \USB$, and so the equality holds throughout.
\end{proof}

To motivate the importance of Proposition~\ref{prop:all_same_sign}, we may consider any min-max robust combinatorial optimization problem with cost uncertainty~\cite{AissiBV09}. Such a problem can be formulated with uncertain objective coefficients and equivalently to fit the formulation~\eqref{prop:genprob} with a single uncertain constraint where $S^+=[n]$. Following Proposition~\ref{prop:all_same_sign}, the robust counterpart of this problem maintains the same form as the nominal problem. In particular, the robust counterpart of a polynomially solvable problem remains polynomially solvable and the existing algorithms for the nominal problem can be used to solve the robust counterpart. We illustrate this numerically on the shortest path problem in Section~\ref{sec:spcomp}. 
}

We now 
consider the case where one of the components of $\delta$ has a positive coefficient and all other components have a 
nonpositive one, or the similar case with opposing signs.

\begin{proposition}\label{prop:one_pos_neg}
Let $\ubar{\delta}$ and $\bar{\delta}$ be defined by \eqref{eq:def_bar} {and let $(x,y)\in\mathcal Z$.} Consider the worst-case problem
\begin{equation}\label{eq:counterpart1}
\max_{\delta\in \USB}\left[ \phi(\delta)\equiv {\delta^\top Cx}
\right],\end{equation}
\begin{enumerate}[ref=\ref{prop:one_pos_neg}(\roman*),label=(\roman*)]
\item If there exists $j\in[n]$ such that $S^{++}=\{j\}$ and $S^-=[n]\setminus\{j\}$,
then, there exists an optimal solution $\delta'$ to \eqref{eq:counterpart1} 
satisfying either (a) $\delta'_j=\bar \delta_j$ and $\delta'_k=\max\{\ubar\delta_k,\bar \delta_j-\dist(k,j)\}$ for all $k\neq j$, 
or (b)
$\delta'_l=\ubar\delta_l$ for some $l\neq j$ {satisfying $\ubar\delta_l<\bar\delta_j-\dist(l,j)$ so that} 
$\delta'_j=\ubar\delta_l+\dist(l,j)$ and  $\delta'_k=\max\{\ubar{\delta}_k,\delta'_j-\dist(k,j)\}$ for all $k\notin \{j,l\}$.\label{prop:one_pos}
\item If there exists $j\in[n]$ such that $S^{--}=\{j\}$ and $S^+=[n]\setminus\{j\}$,
then, there exists an optimal solution $\delta'$ to \eqref{eq:counterpart1} 
satisfying either (a) $\delta'_j=\ubar \delta_j$ and $\delta'_k=\min\{\bar\delta_k,\bar \delta_j+\dist(k,j)\}$ for all $k\neq j$, 
or  (b)
$\delta'_l=\bar\delta_l$ for some $l\neq j$ {satisfying $\bar\delta_l>\ubar\delta_j+\dist(l,j)$ so that} 
$\delta'_j=\bar\delta_l-\dist(l,j)$ and  $\delta'_k=\max\{\bar{\delta}_k,\delta'_j+\dist(k,j)\}$ for all $k\notin \{j,l\}$.\label{prop:one_neg}
\end{enumerate}
\end{proposition}
\begin{proof}
We will prove the result for the case where $S^{++}=\{j\}$ and $S^{-}=[n]\setminus\{j\}$, since the proof of the result for the case where $S^{--}=\{j\}$ and $S^{+}=[n]\setminus\{j\}$ is analogous. Let us consider the function
$$
\theta(\epsilon)=\sum_{k\in [n]} \delta_k(\epsilon) C_k^\top x,
$$
where $\delta_j(\epsilon)=\ubar \delta_j+\epsilon$, and $\delta_k(\epsilon)=\max\{\ubar \delta_k,\delta_j(\epsilon)-\dist(k,j)\}$ for each $k \neq j$. 
{To show} that
$
\max_{\delta\in \USB} \phi(\delta)=\max_{\epsilon\in[0,\bar \delta_j-\ubar \delta_j]} \theta(\epsilon), 
$
{observe} that for {all $l,l'\in[n]$ and every} $\epsilon\in[0,\bar \delta_j-\ubar \delta_j]$, 
\begin{align*}
    \delta_l(\epsilon)-\delta_{l'}(\epsilon)&=\max\{\delta_{j}(\epsilon)-\dist(l,j),\ubar{\delta}_l\}-\max\{\delta_{j}(\epsilon)-\dist(l',j),\ubar{\delta}_{l'}\}\\
    &\leq 
\max\{\delta_{j}(\epsilon)-\dist(l,j)-(\delta_{j}(\epsilon)-\dist(l',j)),\ubar{\delta}_{l}-\ubar{\delta}_{l'}\}\\
&\leq \dist(l,l'),
\end{align*}
{where the second inequality followed from the triangle inequality} {and $\ubar{\delta}\in \USB$ by Lemma~\ref{lem:two_members_ubs}}. So, $\delta(\epsilon)\in\USB$, and therefore  \[\max_{\epsilon\in[0,\bar \delta_j-\ubar \delta_j]} \theta(\epsilon)\leq\max_{\delta\in \USB} \phi(\delta).\] The reverse inequality follows from the observation that, for any $\delta'\in\USB$, defining $\epsilon'=\delta'_j-\ubar\delta_j$, it is evident that 
$\delta'_j=\delta_j(\epsilon')$ and {by feasibility in view of~\eqref{eq:USB_for_3},} $\delta'_k\geq\delta_k(\epsilon')$ for $k\neq j$, so that 
$\phi(\delta')\leq \phi(\delta(\epsilon')) = \theta(\epsilon')$. 

{
Furthermore, note that 
{the coefficient of $\delta_j$ being nonnegative implies that there exists an optimal $\epsilon$ such that} $\epsilon\geq \ubar{\epsilon}\equiv \min\{\min_{k\neq j} \{\ubar{\delta}_k+\dist(k,j)\},\bar{\delta}_j\}-\ubar{\delta}_j$ 
(that is $\ubar{\epsilon}$ is the greatest 
number 
that {satisfies} $\delta_k(\ubar{\epsilon})=\ubar{\delta}_k$ for all
 $k\neq j$).  
}

The function $\theta$ is piece-wise linear so its maximum {over $[\ubar\epsilon,\bar\delta_j-\ubar\delta_j]$} is 
{attained} either at the 
an 
 endpoint of the interval or at one of its breakpoints. 
 {The left endpoint, $\delta(\ubar{\epsilon})$, satisfies either case (a) or case (b), since either $\delta_j(\ubar{\epsilon})=\bar\delta_j$, or there exists $l\neq j$ such that $\delta_j(\ubar\epsilon)=\ubar\delta_{l}+\dist(l,j)$.}  
 The right {endpoint} 
of the interval corresponds to case (1). {Finally,} each of the {other} breakpoints {corresponds} to some $l\in [n]\setminus\{j\}$ in case {(b)}.
\end{proof}

Proposition~\ref{prop:one_pos_neg} characterizes the structure of an optimal solution for \eqref{eq:counterpart1}, resulting in the following straightforward corollary, which enables replacing the robust constraint \eqref{eq:gen_cons}  with $n$ linear constraints. 
\begin{corollary}\label{cor:plus_singleton_RC}
\begin{enumerate}[label=(\roman*)]
\item If $S^{++}=\{j\}$ and $S^-=[n]\setminus\{j\}$, then \eqref{eq:gen_cons} is equivalent to the following set of $n=|V|$ linear inequalities
\begin{align*}
    &\sum_{k\in [n]\setminus\{j\}}\max\{\bar{\delta}_j-\dist(k,j),\ubar{\delta}_k\} C_k^\top x+\bar\delta_jC_j^\top x+d^\top y\leq c &\\
    &\sum_{k\in [n]\setminus\{l\}}\max\{\ubar{\delta}_l+\dist(l,j)-\dist(k,j),\ubar{\delta}_k\} C_k^\top x+\ubar{\delta}_l C_l^\top x +d^\top y\leq c & l\in [n]\setminus\{j\}
\end{align*}
\item If $S^{--}=\{j\}$ and $S^+=[n]\setminus\{j\}$, then \eqref{eq:gen_cons} is equivalent to the following set of $n=|V|$ linear inequalities
\begin{align*}
    &\sum_{k\in [n]\setminus\{j\}}\min\{\ubar{\delta}_j+\dist(k,j),\bar{\delta}_k\} C_k^\top x+\ubar\delta_jC_j^\top x+d^\top y\leq c &\\
    &\sum_{k\in [n]\setminus\{l\}}\min\{\bar{\delta}_l-\dist(l,j)+\dist(k,j),\bar{\delta}_k\} C_k^\top x+\bar{\delta}_l C_l^\top x +d^\top y\leq c & l\in [n]\setminus\{j\}
\end{align*}
\end{enumerate}
\end{corollary}

{The direct dualization of a robust constraint for $\USB$, detailed in~\eqref{eq:robust_reformulation}, results in adding $2(n+|E|)$ variables and $n$ constraints to the original problem. When the conditions of Corollary~\ref{cor:plus_singleton_RC} are satisfied, the latter shows how to provide instead a reformulation involving only $n-1$ additional constraints.}


\shim{One application with constraints exhibiting the structure of mentioned both in Proposition~\ref{prop:all_same_sign} and Corollary~\ref{cor:plus_singleton_RC} is robust radiotherapy treatment planning (RTP) accounting for homogeneity, introduced in~\cite{goldberg2024robust}. In this problem, the intensity of $n_x$ beamlets  
 given by $x\in\R_+^{\nx}$, need to be determined in order maximize the effective dose delivered to a tumor. The influence of the 
beamlets on $N$ {three-dimensional grid elements of the treated region}, called voxels, is given by an influence matrix, 
$D\in\R_+^{N \times n_x }$, and accordingly, the \ngtwo{resulting} voxel dose \ngtwo{vector is} 
the product $Dx$. The effective dose at a tumour voxel $i\in P\subset [N]$, {is then determined by the} radiosensitivity parameter $\delta_i\in[0,1]$, {as} 
$\Delta_i=\delta_i\sum_{j\in[n_x]} D_{ij}x_j.$ However, in practice, the value of $\delta\in \R^n$ is not known and is deduced from \ngtwo{uncertain} imaging data as $\hat{\delta}$. 
Prescribed \ngtwo{dose} bounds, 
\ngtwo{are} of the form $\ubar{\Delta}_i \leq \Delta_i \leq \bar \Delta_i$, 
\ngtwo{so that either $S^+=[n]$ or $S^-=[n]$} 
\ngtwo{and these bounds} can be 
reformulated using Proposition~\ref{prop:all_same_sign}. 
\ngtwo{Other type of} homogeneity \ngtwo{constraints} 
limit the 
ratio $\Delta_j\leq \mu \Delta_k$, for each pair of tumor voxels $j,k\in P$, \ngtwo{so} each of these constraints satisfy $|S^{++}|=|S^{--}|= 1$ and thus can be reformulated using Corollary~\ref{cor:plus_singleton_RC}. 
}
 
\subsection{Column \shim{G}eneration}
\label{sec:CG}


Polytope $\USB$ is defined by $2(|E|+n)$ linear constraints, leading to ${2m(|E|+n)}$ dual variables in {the linearized formulation corresponding to}~\eqref{eq:robust_reformulation}, which may become {a} bottleneck {in solving dense large instances of~\eqref{eq:robust_prob}}
. It is therefore natural to seek a solution algorithm that handles $\USB$ by generating the constraints of the polytope dynamically. This is the purpose of this section. For the sake of simplicity we restrict our attention to linear programs, 
{so that the nominal feasibility set 
of 
$(x,y)$ can be defined as} 
the polyhedron $\Z=\set{(x,y)\in\R_+^{n_x}\times\R_+^{n_y}}{{Fx+Hy}\leq h}$, {with $F\in\R^{\nZ\times\nx}$ and $H\in\R^{\nZ\times\ny}$}. Furthermore, we present our approach for a general uncertainty polytope, and for readability purpose, assume that the problem contains a unique robust constraint, associated to the uncertainty polytope
$\U=\set{\delta\in \R^n}{ \Umat \delta \leq \Urhs}$.
Under the above assumptions, {the robust counterpart of} problem~\eqref{eq:prob} becomes
\begin{subequations}
\label{eq:robustLP}
\begin{align}
&\min_{{x\in\R^{n_x}_+,y\in\R^{n_y}_+}} && {\costx^\top x+ \costy^\top y} \label{eq:robustLP:first}\\
&\text{subject to}&& \delta^\top C x +  d^\top y \leq c, \quad \forall \delta \in \U \\
&&& {Fx + Hy} \leq h 
\label{eq:robustLP:last}
\end{align}
\end{subequations}
Let {$\L$} denote the index set of the {linear} constraints 
{defining} $\U$. As mentioned above, the purpose of the algorithm presented next is to find an optimal solution to the above problem by considering only a 
subset {$\subsetL$ of $\L$}. {Ideally, the subset $\subsetL$ that needs to be considered remains small compared to $\card\L$, which will be the subject of the computational study that follows in Section~\ref{sec:cgcomp}.} 
{Given a subset of indexes $\subsetL\subseteq\L$, we denote {the relaxed uncertainty set} by $\U(\subsetL)=\set{\delta\in \R^n}{\Umat(\subsetL) \delta \leq \Urhs(\subsetL)}$, where $\Umat(\subsetL)$ {is a submatrix of $A$} and $\Urhs(\subsetL)$ {is a} 
subvector of 
$b$ {with rows corresponding to subset $\subsetL$}. 
{Of course, b}y definition, {$\U(\L)=\U$}.
}

We consider next the robust counterpart of~\eqref{eq:robustLP} for the uncertainty set {$\U(\subsetL)$}. Introducing the variable 
{vector $\alpha\in\R^{\card{\subsetL}}_+$}  
dual to the constraints of {$\U(\subsetL)$}, we obtain the \emph{restricted} master problem
\begin{subequations}
\label{eq:primal}
\begin{align}
&\min_{{x\in\R^{n_x}_+,y\in\R^{n_y}_+,\alpha\in\R^{|\tilde{\L}|}_+}} && {\costx^\top x + \costy^\top y} \label{eq:primal:obj}\\
&\text{subject to}&& \alpha^\top\Urhs(\subsetL)+ d^\top y \leq c,\\ 
&&& \alpha^\top\Umat(\subsetL) =  C x,\\ 
&&& {Fx+Hy} \leq h. 
\label{eq:primal:last}
\end{align}
\end{subequations}
The dual linear program of the above problem is
\begin{subequations}
\label{eq:dual}
\begin{align}
&\max_{{\sigma\in\R_+,\mu\in\R^{n},\nu\in\R^{n_z}_+}} && -c\,\sigma - \nu^\top h \label{eq:dual:obj}\\
&\text{subject to}&& - \Urhs(\subsetL)\sigma + \Umat(\subsetL) \mu \leq 0 \label{eq:dual:G}\\ 
&&& - 
{C^\top\mu} - F^\top \nu \leq \costx,
\label{eq:dual:second}\\ 
&&& -d \sigma - H^\top \nu \leq \costy.
\end{align}
\end{subequations}
{Let $(x^*,y^*,\alpha^*)$ and} $(\sigma^*,\mu^*,\nu^*)$ be a pair of optimal solutions to the linear programs~\eqref{eq:primal} and~\eqref{eq:dual}, respectively. 
Therefore, the dual solution $(\sigma^*,\mu^*,\nu^*)$ is feasible for the complete dual problem~\eqref{eq:dual} with $\subsetL=\L$ if and only if
\begin{equation}
\label{eq:separation}
-\Urhs_{\ell} \sigma^* + \Umat_{\ell} \mu^* \leq 0,
\end{equation}
for each $\ell\in {\L\setminus \subsetL}$, which amounts to verifying individually each constraint of $(\Umat,\Urhs)$ indexed by $\L\setminus \subsetL$. 
The following proposition is folklore in the literature 
on dynamic column generation for linear programs. 
\begin{proposition}
Let $(x^*,y^*,\alpha^*)$ and $(\sigma^*,\mu^*,\nu^*)$ be a pair of optimal solutions to the linear programs~\eqref{eq:primal} and~\eqref{eq:dual}. These solutions are optimal for
~\eqref{eq:primal} and~\eqref{eq:dual} 
{with} $\L$ 
{in place} of  $\subsetL$ if {and only if}~\eqref{eq:separation} holds for each $\ell\in \L\setminus \subsetL$. 
\end{proposition}
\vskip-5pt
\begin{algorithm}
	\caption{column generation algorithm (\primalCG)\label{algo:primalCG}}
        Let $\subsetL^0\subseteq\L$ index a starting subset of constraints of $\U$\;
		Initialize $\subsetL\leftarrow \subsetL^0$\;
		\While{a column has been generated}{
        $(\sigma^*,\mu^*,\nu^*)\leftarrow$ optimal solution to~\eqref{eq:dual}\;
        \lFor{$\ell\in {\L\setminus \subsetL}$}
        {\label{algo:primalCF:forloop}
        \textbf{if} \eqref{eq:separation} \emph{is violated} \textbf{then}
        $\subsetL\leftarrow\subsetL\cup\{\ell\}$
		}
	}
\end{algorithm}
This iterative 
scheme naturally leads to either a column generation algorithm for solving the primal problem~\eqref{eq:primal}, where variables $\alpha$ are generated dynamically {(see Algorithm~\ref{algo:primalCG})}, or a constraint generation algorithm for solving the dual problem~\eqref{eq:dual}, where the constraints~\eqref{eq:dual:G} are generated dynamically. In both cases, the separation problem amounts to enumerate each constraint $\ell\in {\L\setminus \subsetL}$ and verify that~\eqref{eq:separation} is satisfied. \

An important aspect of the algorithm lies in the generation of the initial set of constraints $\subsetL^0$, which depends on the uncertainty set considered. We detail next how to generate such a set for $\USB$, 
in particular when the latter is defined by a dense graph $G$, {which is the case in which Algorithm~\ref{algo:primalCG} is more likely to be useful}. While all edges may be necessary in the definition of the polytope, one may hope that a shortest path tree, $T\subseteq E$, based on weights $\gamma$ {and rooted at the center of the graph (according to the distances $\gamma$)}, already provides a good outer-approximation of the polytope. Therefore, the set $\subsetL^0$ we suggest for $\USB$ contains the indexes of the $n$ individual bound constraints, $[n]$, together with the indexes $T$ of the tree for the second set of constraints, specifically
$
\USB(\subsetL^0)=\set{\delta\in \R^{n}}{|\delta_{j}-\hat \delta_j|\leq \gamma_{jj}, \forall j\in [n], \; |\delta_k-\delta_j|\leq \gamma_{kj},\; \forall \{k,j\}\in T}.
$
One can readily apply our results to more than one robust constraint. Eventually, this amounts to check that more dual variables satisfy constraints of the type~\eqref{eq:separation}, {adding the enumeration over the constraints, $i\in[m]$, in the for-loop presented in line~\ref{algo:primalCF:forloop} of Algorithm~\ref{algo:primalCG}. Also in this case, one can choose between keeping track of individual sets $\subsetL_i$ for each $i\in[m]$, or construct a common pool of indexes. Our implementation relies on the former, therefore considering $\subsetL_i$ for each $i\in[m]$.}

{Last, we 
{note a} possible extension} to {integral} $x$ and/or $y$. 
In 
{such a} case, the primal problem~\eqref{eq:primal} is a MILP so the linear programming duality can no longer be applied to this problem. However, when solving a node of the branch-and-bound algorithm that branches on the integrality restrictions of $x$ and/or $y$, one can define a dual similar to~\eqref{eq:dual} and leverage it to generate new variables $\alpha$ through~\eqref{eq:separation}. Eventually, such an idea leads to a branch-and-price algorithm for solving the MILP counterpart of~\eqref{eq:primal}, where variables $\alpha$ are generated dynamically.

\subsection{The Constraint Generation Subproblem}\label{sec:const_generation}

In Section~\ref{sec:comapct_reform} we discussed situations in which the problem structure allows for compact closed-form reformulations of the robust constraints, while in Section~\ref{sec:CG} we proposed a column generation method, which is designed for solving \eqref{eq:robust_reformulation} when the graph $G$ is dense. In this section, we focus on instances where the graph $G$ is sparse \shim{or when a constraint involves a relatively small numbers of components in $\delta$}, and discuss the possible use of a generic constraint generation method. Specifically, we show that solving the adversarial problem for each constraint can be done efficiently. 

\shim{As a motivation for this construction, consider a dynamic extension of the robust RTP problem, mentioned in Section~\ref{sec:cgcomp}, where the dose is administered in $T<<n$ fractions. In this setting, $x_t\in \R^{n_x}_+$ is the beamlet intensity at fraction $t$ and the effective dose in voxel $i\in[n]$ at fraction $t\in[T]$ is given by
$\Delta_{ti}=\delta_{ti}\sum_{j\in[B]} D_{ij}x_{ti}$. 
The constraints can simply be adapted to the fractionated form by {applying them to fractional doses instead}. However, additional constraints on the cumulative dose may apply, including simple bounds  $\ubar{\Delta}_i^c\leq \sum_{t\in[T]} \Delta_{ti}\leq \bar \Delta_i^c$, and cumulative homogeneity constraints take the form  
$\sum_{t\in[T]} \Delta_{ti}\leq \mu^c \sum_{t\in[T]} \Delta_{tk}$. 
In this setting,  the uncertainty is not only spatially connected but also temporally connected. Specifically, for every voxel $i$ and fraction $t$, the value $\delta_{ti}$ will not be significantly different from $\delta_{(t+1)i}$.
The number of uncertain elements involved in each homogeneity constraint is also relatively small ($2T$). Thus, the associated subgraph of $G$ relevant to the constraint has only $O(T)$  
vertices and edges.
}



\shim{Consider the $i$th constraint of~\eqref{eq:robust_constr} (omitting the index $i$), which is equivalent to}
\begin{equation}\label{eq:reformulation_const}\max_{\delta\in\USB} \delta^\top C x\leq b-d^\top y.
\end{equation}
Let $(V,\hat E,\gamma)$ denote the bi-directed graph corresponding to $G$ with $\gamma_{kj}=\gamma_{jk}$. Then, defining $s_j=C_{j}x$, following Proposition~\ref{prop:eq_bar}
 the problem  can be written as 
\begin{subequations}
\label{form:primalsep}
\begin{align}
& \max_{\delta\in \R^n} && \sum_{j\in[n]} s_j\delta_j\\
& \text{subject to} && \delta_k-\delta_j\leq \gamma_{kj} && (k,j)\in \hat E\\
& && 
\ubar\delta_j \leq \delta_j 
\leq 
\bar\delta_j && j\in [n]. 
\end{align}
\end{subequations}
Defining a dummy node $d$, we can formulate  the dual problem of~\eqref{form:primalsep} as
\begin{subequations}\label{form:dualsep}
\begin{align}
& \min_{z\geq 0} && \sum_{(k,j)\in \hat E}\gamma_{kj}z_{kj} +\sum_{j\in [n]}
\bar\delta_jz_{dj} -
\sum_{j\in [n]}
\ubar\delta_j
z_{jd}\label{eq:dual_obj}\\
& \text{subject to} && \sum_{k:(k,j)\in \hat E}z_{kj}-\sum_{k:(j,k)\in \hat E}z_{jk}
+z_{dj}-z_{jd}
= s_j && j\in [n].\label{constr:flowineq}
\end{align}
\end{subequations}
Accordingly, we consider a 
directed graph with the additional vertex $d$ and additional arcs $(j,d)$ and $(d,j)$ for each $j\in[n]$, where edge set $E'=\hat E\cup\bigcup_{j\in V}\{(j,d),(d,j)\}$. 

{Although not immediately apparent, problem~\eqref{form:dualsep} is, in fact, a minimum-cost flow problem. Further, although the  objective~\eqref{eq:dual_obj} includes negative cost coefficients, the corresponding graph does not contain any negative cycles, as established by the next proposition.

\begin{proposition}~
\begin{enumerate}[label=(\roman*)]
\item \label{propitem:flowprob} Problem~\eqref{form:dualsep} is a minimum-cost flow problem in graph $G'=(V\cup\{d\},E')$.
\item \label{propitem:nonnegcost} Problem \eqref{form:dualsep} can be rewritten as a minimum cost flow problem with nonnegative cost coefficients (and no negative cycles),  given by
\begin{align}
& \min_{z\geq 0}  \set{\sum_{(k,j)\in \hat E}(\gamma_{kj}+\ubar{\delta}_k-\ubar{\delta}_j)z_{kj} +\sum_{j\in [n]}(\bar\delta_j-\ubar\delta_j)z_{dj} +
\sum_{j\in [n]}
\ubar\delta_j s_j}{\eqref{constr:flowineq}}\label{eq:dual_obj_new}
\end{align}
%
\end{enumerate}
\end{proposition}}
\begin{proof}~
\ref{propitem:flowprob}
The claim follows by showing that appending the flow conservation constraint for node $d$
\begin{equation*}
\sum_{j\in[n]}z_{jd}-\sum_{j\in[n]} z_{dj} = -\sum_{j\in [n]} s_j,  \label{constr:flow_cons_d}
\end{equation*}
retains all feasible solutions of~\eqref{form:dualsep}. This follows from 
summing over all constraints~\eqref{constr:flowineq}. 

\ref{propitem:nonnegcost}
Equalities \eqref{constr:flowineq} are equivalent to
$$z_{jd}=\sum_{k:(k,j)\in \hat E}z_{kj}-\sum_{k:(j,k)\in \hat E}z_{jk}
+z_{dj}-s_j,$$ 
{the substitution of which}  into 
\eqref{eq:dual_obj}
results in \eqref{eq:dual_obj_new}.
Finally, due to Lemma~\ref{lem:two_members_ubs} we know that for all $\{j,k\}\in E$
$|\ubar{\delta}_j-\ubar{\delta}_k|\leq \dist(j,k)\leq \gamma_{jk}$
implying that for all $(j,k)\in \hat{E}$ the coefficient of $z_{jk}$ is nonnegative. 
Moreover, since for all $j\in [n]$ $\bar{\delta}_j\geq \ubar{\delta}_j$ the coefficients of $z_{dj}$ are also nonnegative. {By part~\ref{propitem:flowprob},~\eqref{form:dualsep} is a minimum-cost flow problem, rewritten as~\eqref{eq:dual_obj_new} with all cost coefficients being nonnegative, concludes} the proof.
\end{proof}


Thus, 
the adversarial problem \eqref{form:primalsep} can be solved 
as the minimum cost flow 
problem with nonnegative arcs given by~\eqref{eq:dual_obj_new}.
The solution of~\eqref{eq:dual_obj_new} with nonnegative cost coefficients, may allow the use of faster algorithms that make such an assumption. Indeed, minimum cost flow can be solved using dedicated methods, including ones that are strongly polynomial; see for example~\cite{Orlin1993}. Although mostly a theoretical result, there are also cases in which these algorithms outperform general LP methods in practice, but this is beyond the scope of the current paper.

\section{Evaluating the Conservatism of $\USB$}\label{sec:numerics}
\shim{We present three computational studies to illustrate the practical relevance of our proposed smooth uncertainty model. 
In these studies we explore both the size of the uncertainty set obtained, as well as the quality of the solutions on examples with both synthetic and real data. }

\subsection{Experiments using Probabilistic \shim{B}ounds}
\label{sec:res:probabounds}

We illustrate the quality of the solutions returned by the multi-location transshipment problem inspired by~\cite{herer2006multilocation}. The \shim{construction of $\USB$ \ngtwo{is designed to be} in line with} the probabilistic bounds from Section~\ref{sec:probbounds}, comparing them with the solutions returned by the robust models using the box, the rotated box, and the ellipsoid uncertainty set. We consider a multi-location transshipment problem with one supplier and $n$ nonidentical retailers, each having a distinct stocking location, facing uncertain customer demands. The system inventory is reviewed periodically, and replenishment orders are placed with the supplier. \shim{At} any time period, transshipments provide a means to reconcile demand-supply mismatches. 

The model below considers two stages of decision. In the first stage we decide on ordering $y^O_i$ for each retailer $i$, which defines the starting inventory level 
\ng{of} retailer $i$. In the second stage, after the demand $\delta_i$ 
\ng{is} revealed, retailer $i$ can satisfy 
 \ng{the} demand from the purchased stock and by \ng{the} transship\ng{ped} 
 amounts $x_{a}$ (which may depend on the demand) at cost $c^{TR}_{a}$, where $a\in \mathcal{A}_{i}^{in}$. After the transshipments occur the retailer pays $c^h_j$ for each unit of excess stock 
 \ng{held} and $c^b_\shim{i}$ for 
 \ng{each} unit of \ng{unfulfilled} demand 
 (backlogging cost). 
Assuming 
uncertain demand
, in the robust version of this problem we wish to protect ourselves against demands 
defined in an uncertainty set $\U$, leading to problem 

\begin{subequations}
\begin{align}    &\min_{\substack{y^O\in\R^n_+,\\x_a(\cdot):\U\rightarrow \R,\;a\in\mathcal{A}\\ \tau_j(\cdot):\U\rightarrow\R,\;j\in[n]}} && \epicost &&\nonumber\\
 &   \text{subject to} && \epicost\geq \sum_{i\in[n]} c^O_i y_i^O+\sum_{i\in[n]} \tau_i(\delta)+\sum_{a\in\A} {c}^{TR}_ax_a(\delta),&& \forall \delta\in\mathcal{U}\nonumber\\
    &&&\tau_i(\delta)\geq c^b_i\left(\delta_i-y^O_i-\sum_{a\in \A^{in}_i} x_{a}(\delta)+\sum_{a\in \A^{out}_i} x_{a}(\delta)\right),&& i\in[n],\;\forall \delta\in\mathcal{U}\label{eq:taub}\\ 
    &&&\tau_i(\delta)\geq c^h_i\left(y^O_i-\delta_i+\sum_{a\in \A^{in}_i} x_{a}(\delta)-\sum_{a\in \A^{out}_i} x_{a}(\delta)\right), &&i\in[n],\;\forall \delta\in\mathcal{U}\label{eq:tauh}\\
    &&&x_{a}(\delta)\geq 0, && a\in \A,\;\forall \delta\in\mathcal{U},\nonumber   
\end{align}
\end{subequations}
which we address by restricting the second-stage variables $x_{a}(\delta)$ $\tau_i(\delta)$ to affine decision rules
\begin{align*}
x_{a}(\delta)=y^{TR}_{a}+\sum_{j\in[n]} \delta_jx^{TR}_{aj},\quad a\in \A,\quad\mbox{and}\quad
\tau_i(\delta)=y^{IC}_i+\sum_{j\in[n]}\delta_j x^{IC}_{ij},\quad i\in[n].
\end{align*}

 The instances used to compare the algorithms are inspired by those of~\cite{herer2006multilocation}. 
The retailers $[n]$ are connected together in a bi-directional 
cycle, and the supplier is connected to all other retailers in both directions. All costs are uniformly distributed between $0$ and $1$, apart from $c^b_\shim{i}$, which is uniformly distributed between $0$ and $4$. Additionally, $\hat\delta_j$ is uniformly distributed between 50 and 150. 

To construct the uncertainty sets,
\shim{we assume the mean vector $\mu$ and positive-definite covariance matrix $\Sigma$ are given, }
and construct the polytope $\USB$ obtained by setting $\gamma_{ij}$ as in Corollary~\ref{cor:gengaussprobbound} for probability $p\in\{0.01,0.05,0.1,0.2,0.3,0.4,0.5\}$. Then, we consider the robust problem for $\delta'\in\USB$ where $\delta'$ is the standardized random vector defined by $\delta'_i=\frac{\delta_i-\mu_i}{\sigma_{i}}$ for each component $i$, in line with Assumption~\ref{ass:mean_variance}, defining also the resulting covariance matrix by $\Sigma'$. In our experiment, each component $\mu_i$ is independently and uniformly distributed in $[1,2]$, while $\Sigma=RR^\top$ with each component of $R$ being independently and uniformly distributed in $[0,1]$.

We assess an optimal solution $(x^*,y^*)$ to the robust transshipment problem corresponding to each value of $p$ by generating independently $10^4$ demand vectors following a distribution $\mathcal{N}(\mu,\Sigma)$, truncated to non-negative values. For each such demand vector $\delta$, we compute the cost as
\begin{equation}
\label{eq:simulated_cost}
\sum_{i\in[n]} c^O_i y^{O*}_i+\sum_{a\in\A} \hat{c}^{TR}_ax^*_a(\delta) + \sum_{i\in[n]}\tau(\delta),
\end{equation}
where $x^*_a(\delta)$ is obtained by using the optimal affine decision rules $y^{TR*}_a$ and $x^{TR*}_{aj}$, while $\tau_i(\delta)$ considers the exact inventory cost, obtaining by taking the smallest value that satisfies~\eqref{eq:taub} and~\eqref{eq:tauh}. This leads to $10^4$ cost realizations for $\USB$, for which we report the mean and worst-case values. \mp{Figure~\ref{fig:costs_proba} reports the average results of this experiment, creating 20 random instances for each probability value $p$ and uncertainty set, and reporting the means of the average and maximum costs over the $10^4$ scenarios.}

\begin{figure}
\centering
\includegraphics[width=0.5\linewidth]{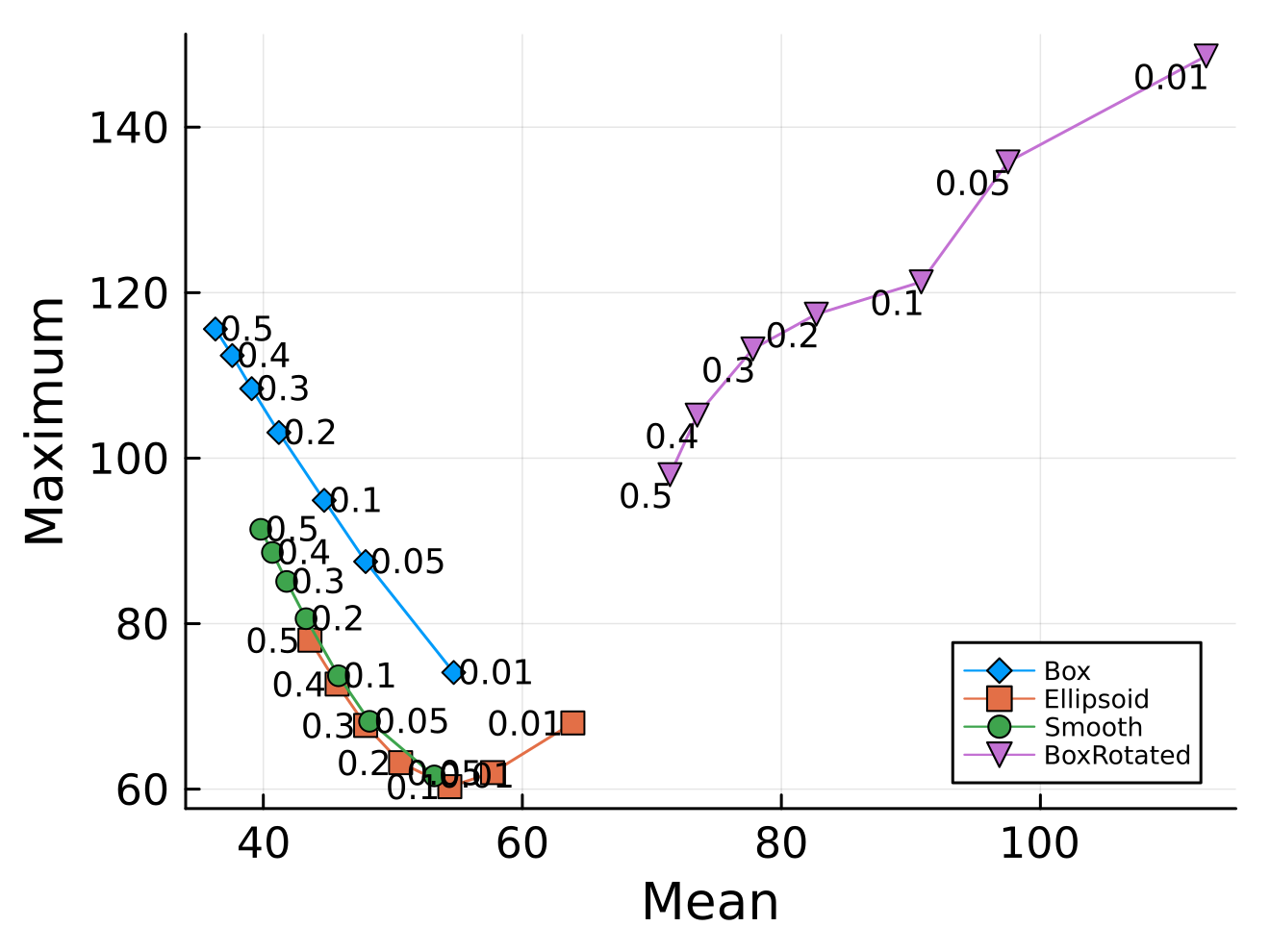}
\caption{Mean and maximum of costs~\eqref{eq:simulated_cost}. \label{fig:costs_proba}
}
\end{figure}
Observe that Figure~\ref{fig:costs_proba} additionally reports the costs obtained for the solutions of other robust models, namely:
\begin{itemize}
\item the box $\U_{box}$ obtained \mp{by considering only the individual bounds $|\delta'_i|\leq \gamma_{ii}$, computing $\gamma_{ii}$ along the lines of Corollary~\ref{cor:gengaussprobboundRB},}
\item the rotated box $\U_{RB}=\{\delta'\in\R^n: \norm{(\Sigma')^{-1}\delta'}_\infty\leq r\}$, setting $r$ as prescribed in Corollary~\ref{cor:gengaussprobboundRB},
\item the ellipsoid $\U_{\mathcal{E}}=\{\delta'\in\R^n: (\delta')^\top(\Sigma')^{-1}\delta\leq \chi^2_{1-p}(n)\}$.
\end{itemize}
Figure~\ref{fig:costs_proba} reports the results for the 4 uncertainty sets above, \mp{with the associated probability $p$ provided as a label of each point}. Given 
that the model using $\U_{RB}$ generated solutions with very high mean and worst-case values compared to the other methods. \mp{Regarding the means, $\Ubox$ obtains the lowest values, followed by $\USB$, and then by $\UE$. The situation is different for maximum values, for which $\UE$ yields the lowest values, closely followed by $\USB$, while $\Ubox$ exhibits significantly higher maximum values. Overall, sets $\USB$ and $\U_{\mathcal{E}}$ mostly dominate $\U_{box}$, the comparison among these two sets being much less contrasted. Therefore, the results indicate that $\USB$ can provide a good linear alternative to $\U_{\mathcal{E}}$, which typically leads to easier optimization problems.}

\subsection{Evaluation of  Uncertainty Sets on Energy Data}
\shim{In this section, we compare between the smooth and ellipsoidal uncertainty sets constructed from real data. We consider the energy consumption 
dataset containing the monthly energy consumption in a 1 square km grid of major cities in China\ngtwo{~\cite{yan2024monthly}.} 
\ngtwo{Using this dataset we} explore the benefits of using the smooth uncertainty set when the amount of data is insufficient to reliably capture correlations. 

We consider constructing two types of uncertainty sets based on the data: the smooth uncertainty set, \mp{based on the maximum deviations \shim{within} the training set}, and the ellipsoid, \mp{based on the covariance \ngtwo{in} the training set.} 
Specifically, given a training set indexed by $\ell$, we computed the mean, the middle and the width of the range of the energy value for each pixel $i$ as $\mu^\ell_i$, $\hat{\delta}^\ell_i$ and $w_{ii}^\ell$, respectively. 
We additionally computed the maximum difference between each pair of pixels $i,j$ as $w^\ell_{i,j}$, and the sample covariance matrix $\Sigma^{\ell}$ and the diagonal correction matrix $D^{\ell}$.
 The smooth uncertainty set $\USB^{\ell,\alpha,\beta}$ is constructed using parameters $\hat{\delta}^{\ell}$ and 
$\gamma_{ii}=\alpha w^{\ell}_{ii},\; \gamma_{ij}=\beta w^{\ell}_{ij},$
where $\alpha$ and $\beta$ can be tuned.
The \mp{ellipsoid} uncertainty set, $\mathcal{U}^{\ell,\rho,\Omega}_{\E}$ is constructed as
$\mathcal{U}^{\ell,\rho,\Omega}_{\E}=\{\delta: (\delta-\mu^\ell)(\tilde{\Sigma}_{\rho}^{\ell})^{-1}(\delta-\mu^\ell)\leq \Omega\},$
where $\tilde{\Sigma}_{\rho}^{\ell}=(1-\rho)\Sigma^{\ell}+\rho D^{\ell}$ is the corrected covariance matrix (needed due to rank deficiency in $\Sigma^{\ell}$), and both $\Omega$ and $\rho\in(0,1)$ can be tuned.


 We compare between the uncertainty sets using $L$-fold cross-validation for $L=7$. The full details of the experiment, summarized here, can be found in Appendix~\ref{appx:experiment_detail}. 
 For every training-validation split $\ell$ and parameter choice, 
 we construct 
 $\USB^{\ell,\alpha,\beta}$ and $\U_\E^{\ell,\rho,\Omega}$ 
 and compute 
$p^{\ell}_{\U}$, the 
\ngtwo{proportion} of the validation set that  
\ngtwo{is contained} within the uncertainty set $\U\in \{\USB^{\ell,\alpha,\beta},\U_\E^{\ell,\rho,\Omega}\}$. For each set of parameters we average this \ngtwo{proportion} 
over $\ell\in [L]$.
Additionally, we want to compare the size of the uncertainty sets. Unfortunately, contrary to the results presented in Section~\ref{sec:probbounds}, computing the volume and diameter of the resulting smooth uncertainty sets is computationally difficult due to their high dimension. Thus, to relate the sizes of the different uncertainty set, we calculate
surrogate relative parameter $\alpha^{\rho,\Omega}$, $\beta^{\rho,\Omega}$ for the ellipsoidal uncertainty set, which are linear functions of $\Omega$. In Figure~\ref{fig:Prob_vs_alpha} we present the probabilities $p_{\USB}^{\alpha,\beta}$ and  $p_{\U_\E}^{\rho,\Omega}$ vs. $\alpha$ and $\alpha^{\rho,\Omega}$, respectively, and Figure~\ref{fig:Prob_vs_beta}, to display these probabilities vs. $\beta$ and $\beta^{\rho,\Omega}$.

\begin{figure}
\centering
\subfloat[Probability of being inside the set vs. $\alpha$.\label{fig:Prob_vs_alpha} ]{\includegraphics[width=0.5\linewidth]{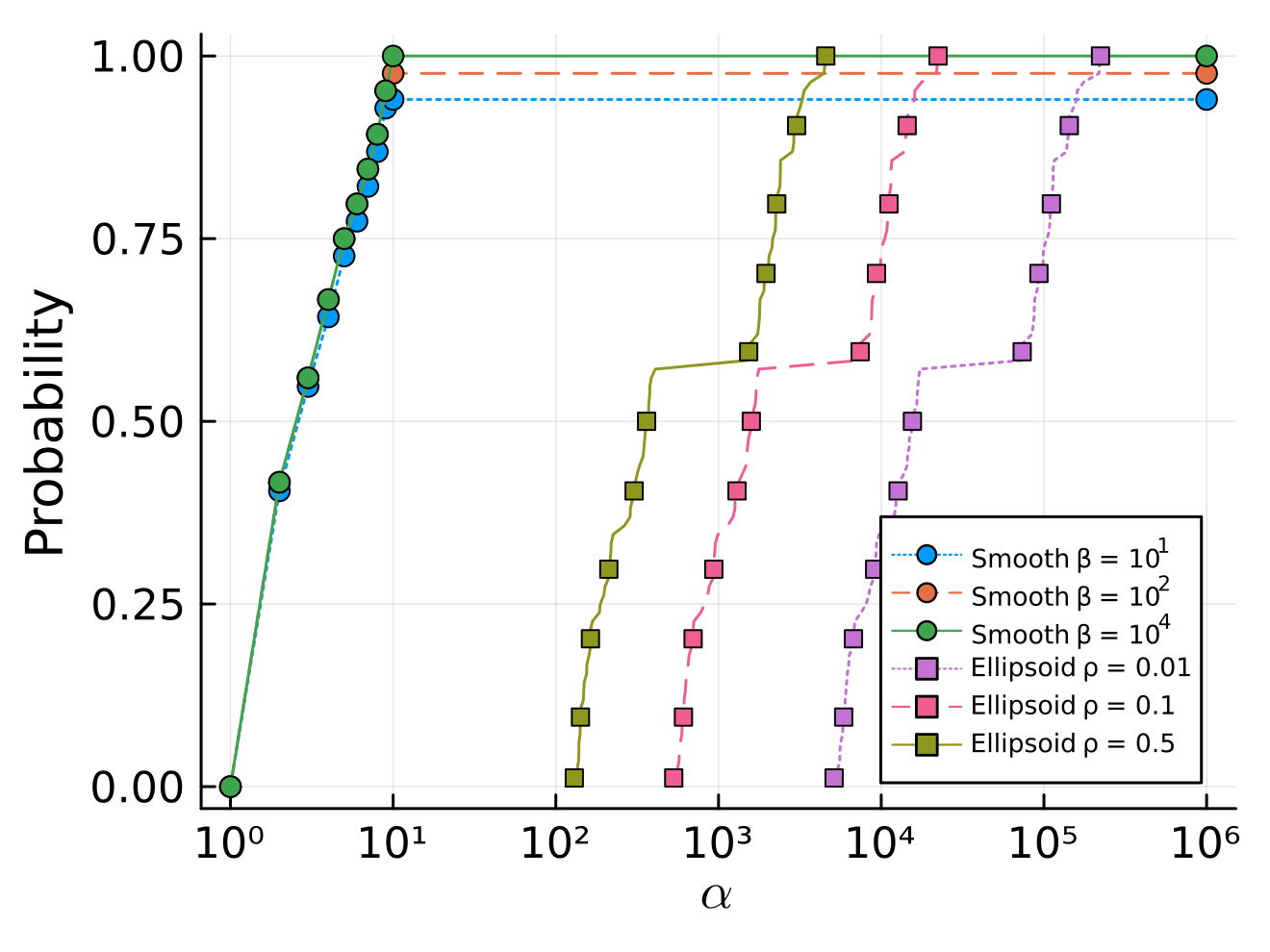}}
\subfloat[Probability of being inside the set vs. $\beta$.\label{fig:Prob_vs_beta} ]{\includegraphics[width=0.5\linewidth]{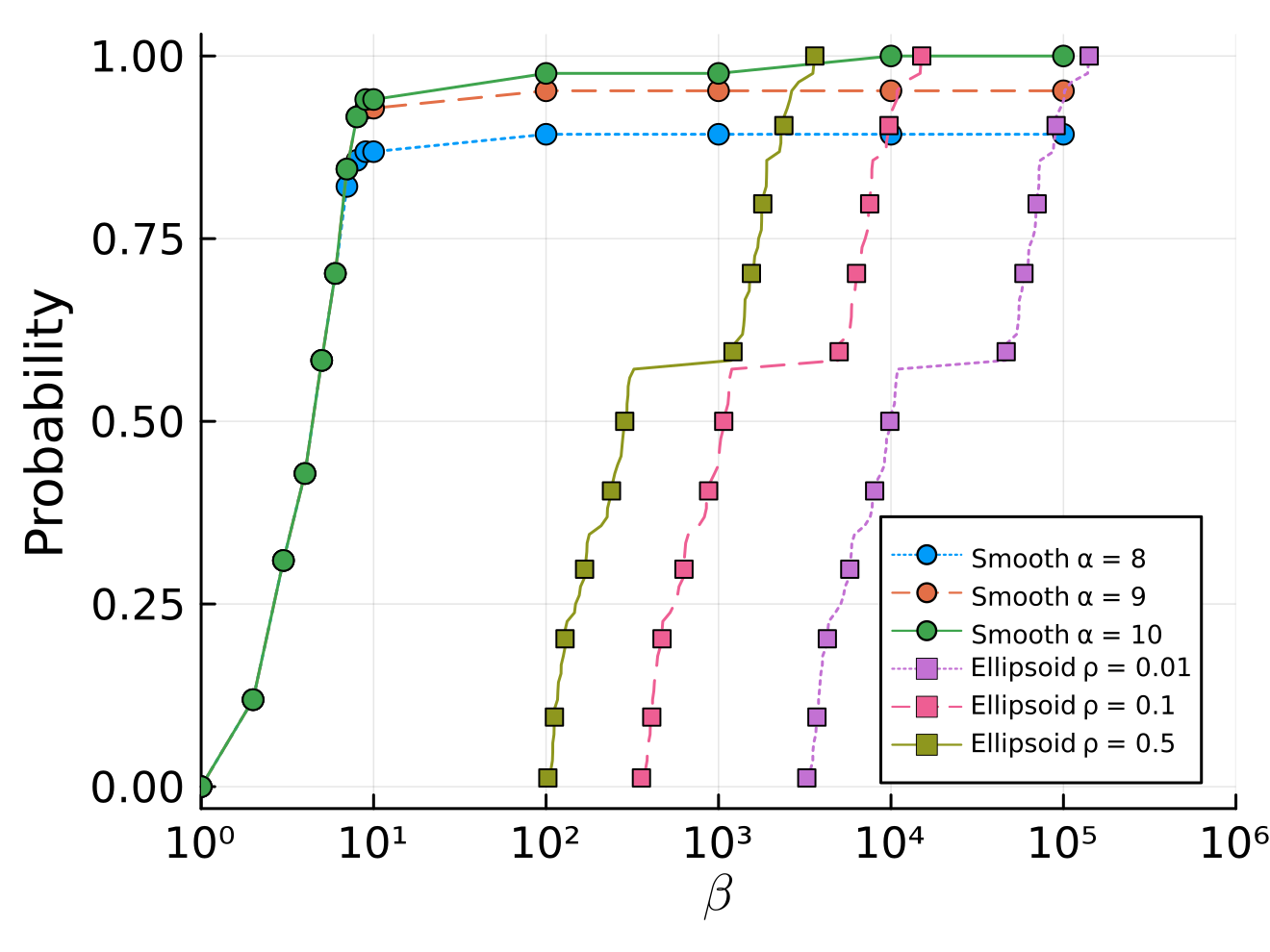}}
\caption{Size-probability tradeoff comparison between $\USB$ and $\U_\E$ uncertainty sets on energy consumption data in China. \label{fig:probability_vs_set_size}
}
\end{figure}

Figure~\ref{fig:probability_vs_set_size} demonstrates that the size of the smooth uncertainty set needed to obtain a probability of at least 95\% to contain out-of-sample points is smaller than that of the ellipsoidal uncertainty set. Indeed, for $\alpha=10$ and $\beta=100$ the smooth uncertainty set provides such a probability, while the ellipsoidal uncertainty set is larger by one to two orders of magnitude in these parameters. This is due, in part, to the small amount of data compared to dimension of the uncertainty. This leads to a low rank covariance matrix that does not allow to fully capture the dependencies between the different energy consumption values. Moreover, the correction of the covariance matrix by $\rho$ ensures the corrected matrix has full rank.  However, as $\rho$ increases the estimated correlation of the original covariance matrix goes to zero. 

Thus, \mp{our results illustrate} that in situations where the amount of data is small relative to dimension of the uncertainty, the smooth uncertainty set may be more appropriate than a covariance based uncertainty set, \mp{such as the ellipsoid}. In this settings, smooth uncertainty sets can avoid degeneracy issues stemming from the low rank covariance matrix, thus producing a smaller (less conservative) set with similar empirical probabilistic guarantees.}

\subsection{\ng{Robust} Shortest \shim{P}ath} 
\label{sec:spcomp}

Given a directed weighted graph $\mathcal{G}=(\mathcal{V},\mathcal{A},\tilde{f})$, with arc set $\mathcal{A}$   \ngtwo{(where $\mathcal{A}=n$), uncertain arc transit times $\tilde f$, 
with ${f}_a=\delta_a\hat{f}_a$ for $\hat{f}_a$ that is the nominal value of ${f}$ (say its mean) and $\delta_a\in \R_+$ that is the relative deviation from the mean, for each $a\in\A$,  the robust shortest path problem 
 is formulated as}
\ngtwo{\begin{equation}\label{form:robsp}
\min_{x\in \X}\max_{\delta\in \mathcal U} \sum_{a\in\A}\delta_a\hat f_a x_a,
\end{equation} where  ${\mathcal{X}}\subseteq[0,1]^{n_x}$
  is a unit flow conservation polytope in $\mathcal{G}$ 
 from some origin $s\in\mathcal V$ to destination $t\in\mathcal V$. This problem has been extensively studied with uncertainty set $\mathcal U$ defined as a box, budget~\cite{bertsimas2003robust}, and ellipsoid~\cite{ghaoui2003worst,bertsimas2004,Chassein2019}. Under ellipsoidal uncertainty, while NP-hard,  small to moderately sized instances of this problem can be solved as a mixed-integer second-order cone program (MISOCP). \ngtwo{More tractable, at least in practice is to model the uncertainty as an axis-aligned ellipsoid; see for example,~\cite{Chassein2019}.}  
 \shim{The rotated box uncertainty set, for which we have also derived probabilistic bounds, provides an alternative for capturing correlations in the data and results in an MILP formulation.}
However, it appears to be highly intractable. 
Theoretically, its intractability is evident as it generalizes the discrete scenario robust shortest path problem, that is NP-hard already with two distinct 
scenarios~\citep{kouvelis2013robust}. In practice we have not been able to solve moderately sized instances of this problem with state-of-the-art integer programming solvers. 

In sharp contrast to the integer programming models resulting from the ellipsoidal and rotated box uncertainty sets, with a smooth uncertainty set, problem~\eqref{form:robsp} can be rewritten as
\begin{equation*}
\min_{(x,y)\in\Z}\set{y}{\sum_{a\in\A} \delta_a\hat{f}_a x_a-y\leq 0},
\end{equation*}
where $\Z=\X \times \R$. In this case, $S^{+}=\mathcal{A}$ and $S^{--}=\emptyset$ 
\shim{that by Proposition~\ref{prop:all_same_sign} leads to a simple robust formulation (for any edge set $E\subseteq\A\times\A$ of the uncertainty set graph $G$)} of the form
\begin{equation}\label{form:rob_sp_smooth}
\min_{x\in {\mathcal{X}}}\sum_{a\in\A}  \bar{\delta}_a\hat{f}_a x_a,
\end{equation} 
where $\bar{\delta}_a$ 
can be efficiently precomputed prior to solving the robust problem, \shim{with complexity of an all-pairs shortest path problem.} Thus, the robust problem has the same structure as the deterministic problem 
\ng{so it} can be solved using 
\ng{similar} \shim{polynomial time} algorithms. 

\begin{figure}[t]
    \centering
    \includegraphics[width=0.55\linewidth]{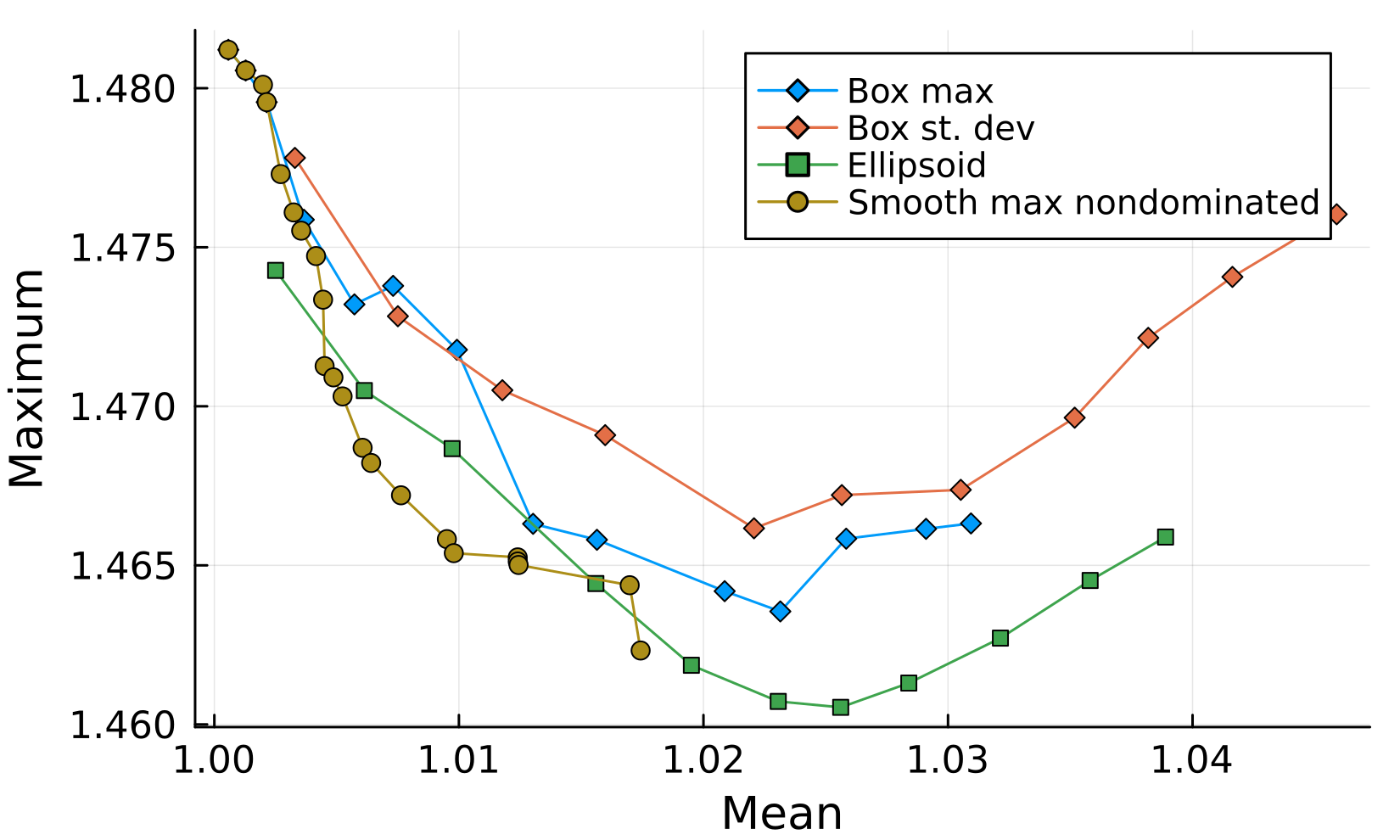}
    \caption{Out-of-sample 
    \shim{trade-off of normalized worst-case and mean path travel times}
    on the 
    Chicago traffic morning dataset~\cite{Chassein2019}. Travel times for each source-destination pair are normalized with respect to the mean nominal model travel time for that pair.}
    \label{fig:seqejordata}
\end{figure}

Our experiment is based on the Chicago morning data from~\cite{Chassein2019}. The experiment involves 10 repetitions of a train-test split of 80\%-20\% of the data, where 20\% of the travel time data are sampled to be used as a test set (out-of sample) and the remaining scenarios are used as a training set for training 200 paths (same paths as studied by~\cite{Chassein2019}).  
The out-of-sample (test) performance 
is displayed in Figure~\ref{fig:seqejordata}, 
comparing the solutions obtained by the robust shortest path problem with smooth uncertainty set~\eqref{form:rob_sp_smooth}, \ngtwo{axis-aligned} ellipsoidal uncertainty set 

 \ngtwo{$\UE$}, where $\Sigma$ is estimated to be a diagonal matrix, and box uncertainty set \ngtwo{whose parameters are estimated based on the data following two different schemes, a maximum based and standard deviation based one; see Appendix~\ref{appx:experiment_detail} for the details.}} 
 . \ngtwo{The smooth uncertainty set curve shows all nondominated points with respect to the additional tunable parameter in the maximum-based setting.} 

Each data point in Figure~\ref{fig:seqejordata} corresponds to the average for the 200  
source-destination pairs (as in the experiments of~\cite{Chassein2019})
over the 10 test sets. 
The horizontal coordinate of each point indicates the average travel time over the test scenarios, and the vertical coordinate of each point reflects the maximum travel time over the same scenarios. 
The paths computed by the maximum-based box uncertainty model perform better than the standard deviation-based ones on the test data. The smooth uncertainty model~\eqref{form:rob_sp_smooth} improves on the performance of both box models and matches the test performance of the ellipsoid model in up to 1.02 times the mean nominal performance, and in some cases even outperforms it. This is while the ellipsoid model requires the solution of an MISOCP while~\eqref{form:rob_sp_smooth} has the same running time complexity of a standard shortest path problem \shim{(after preprocessing).} 


\section{Numerical Study of the \shim{C}olumn \shim{G}eneration \shim{A}lgorithm}\label{sec:cgcomp}

The purpose of the experiments presented next is to assess the potential interest of the dynamic constraint generation of $\USB$ presented in Section~\ref{sec:CG}. The experiments are run on an Intel Xeon E312xx (Sandy Bridge) using CPLEX 20.1 with default parameters. We consider instances generated as in Section~\ref{sec:res:probabounds}, with $n$ ranging in $\{10,15,20,40,60,80,100\}$ and 10 randomly generated instances for each value of $n$.

We compare three solution methods: the full dualization of the robust constraints, in which we consider the entire set of constraints in the definition of the uncertainty set (\compact), the dynamic generation of variables in the primal model described in Section~\ref{sec:CG} (\primalCG), and the dynamic generation of robust constraints directly in the model (\cutgen), which we further describe next. For each $i\in[m]$, we define the finite set of scenarios $\tU_i\subset\USB$.
{These sets together form} the tuple $\tU=(\tU_1,\ldots,\tU_m)$ and {let  $M^\cutgen(\tU)$ denote} the \emph{relaxed} master problem
~\eqref{eq:robust_prob} {where $\USB$ in} 
~\eqref{eq:robust_constr}{, for each $i\in [m]$, is replaced by 
$\tU_i$}.
Following these notations, \cutgen is further described in Algorithms~\ref{algo:cutgen}.

\begin{algorithm}
	\caption{cutting plane algorithm (\cutgen)\label{algo:cutgen}}
		Initialize $\tU_i=\{\hat\delta\}$\;
		\While{a cutting plane has been generated}
        {
        $(x^*,y^*)\leftarrow$ optimal solution to $M^\cutgen(\tU)$\;
        \For{$i\in[m]$}
        {
        $\delta^*\leftarrow$ optimal solution to $\omega^* = \max\limits_{\delta\in\USB}\delta^\top C^i x^* + d^\top_iy^*$\;
        \lIf{$\omega^* > b_i$}
        {
        $\tU_i\leftarrow\tU_i\cup\{\delta^*\}$
        }
        }
		}
\end{algorithm}

We report the numerical results in Tables~\ref{tab:small} and~\ref{tab:large}, respectively focusing on small and larger instances. The tables present the (rounded) average of the values for each set of 10 instances, where times are given in seconds (and include also the standard deviations). Columns ``rounds'' report the number of rounds that generate cuts and variables for \cutgen and \primalCG, respectively. The column ``vars start'' provides the number of variables present in the model after the shortest arborescence has been considered in the relaxed master problem.

Table~\ref{tab:small} focuses on the small instances, showing how \cutgen scales poorly with the dimension of the instances, due to its bad convergence as it generates more cuts than present in the dualized models. Table~\ref{tab:large} then considers the larger instances. Table~\ref{tab:large} reports that \compact appears to be a bit slower than \primalCG, and more importantly, that \mp{it leads models that are much smaller than those handled by \compact, since the model solved at the end of \primalCG considers roughly 10\% of the total number of variables. This aspect is important when memory is an issue, which we illustrated by using a VM with 8GB of memory. The full approach \compact hits memory with $n=100$ already, while \primalCG can run instances with $n$ being as large as 150.}

\vskip-10pt
\begin{table}[h!]
\centering
\setlength{\tabcolsep}{0.35em}
\begin{tabular}{c|rrr|rrr|rrrr} 
\hline
$n$ & \multicolumn{3}{c|}{\compact} & \multicolumn{3}{c|}{\cutgen} & \multicolumn{4}{c}{\primalCG} \\
 & \multicolumn{1}{c}{time} & \multicolumn{1}{c}{vars} & \multicolumn{1}{c|}{cons} & \multicolumn{1}{c}{time} & \multicolumn{1}{c}{rounds} & \multicolumn{1}{c|}{cuts} & \multicolumn{1}{c}{time} & \multicolumn{1}{c}{rounds} & \multicolumn{1}{c}{vars start} & \multicolumn{1}{c}{vars gen} \\
\hline
10 & 0.3 $\pm$ 0.1 & 6787 & 627 & 9.6 $\pm$ 1.3 & 47 & 1601 & 0.4 $\pm$ 0.3 & 4 & 2683 & 428\\
15 & 1.1 $\pm$ 0.1 & 22032 & 1392 & 149.6 $\pm$ 17.1 & 62 & 3395 & 1.1 $\pm$ 0.2 & 4 & 6198 & 1185\\
\hline
\end{tabular}
\caption{Comparison of \compact, \cutgen, and \primalCG on small instances. Time in seconds.\label{tab:small}}
\end{table}

\begin{table}[h!]
\centering
\begin{tabular}{c|rrr|rrrr} 
\hline
$n$ & \multicolumn{3}{c|}{\compact} & \multicolumn{4}{c}{\primalCG} \\
 & \multicolumn{1}{c}{time} & \multicolumn{1}{c}{vars} & \multicolumn{1}{c|}{cons} & \multicolumn{1}{c}{time} & \multicolumn{1}{c}{rounds} & \multicolumn{1}{c}{vars start} & \multicolumn{1}{c}{vars gen} \\
\hline
20 & 2.4 $\pm$ 0.6  & 51 & 2 & 1.5 $\pm$ 0.2  & 4.4 $\pm$ 0.7 & 11 & 3.3 $\pm$ 0.4\\
40 & 18.5 $\pm$ 1.6 & 397 & 10 & 12.0 $\pm$ 1.3  & 6.4 $\pm$ 0.5 & 46 & 30.1 $\pm$ 3.1\\
60 & 63.1 $\pm$ 8.0  & 1325 & 22 & 52.2 $\pm$ 7.4 & 7.4 $\pm$ 0.8 & 103 & 98.5 $\pm$ 5.6\\
80 & 179.0 $\pm$ 8.7 & 3123 & 39 & 153.9 $\pm$ 20.3  & 8.0 $\pm$ 0.7 & 184 & 279.8 $\pm$ 19.5\\
100 & 502.9 $\pm$ 28.8 & 6080 & 60 & 427.8 $\pm$ 33.0 & 8.2 $\pm$ 0.9 & 288 & 614.9 $\pm$ 23.9\\
\hline
\end{tabular}
\caption{Comparison of \compact and \primalCG on larger instances.Time indicated in seconds. Numbers of variables and constraints are expressed in thousands.\label{tab:large}}
\end{table}

\section{Summary and Conclusion}\label{sec:conclutions}

In this paper, we have introduced the smooth uncertainty set, a new polyhedral uncertainty set that enforces bounds on 
the difference between 
pairs of uncertain parameters. This set can be constructed from a weighted undirected graph, with edges that indicate which pairs of uncertain parameters must take close values, and weights that indicate the bounds on the value and differences of these parameters. These weights can be 
determined, for instance, based on historical data, expert knowledge, or correlations when the latter are available. 
Given 
the correlation matrix, we have shown how to set
our model weights to guarantee that the feasibility to the robust constraint implies the feasibility for the associated probabilistic constraint,
for any given probability. 
We have developed efficient methods for solving the resulting robust optimization problems: a custom compact reformulation that improves upon the classical dualization-based approach when the robust constraint possesses a particular structure, and a column-generation algorithm tailored to the smooth uncertainty set for handling more general settings.
The latter method is numerically shown to outperform the classical dualization in 
We have also 
demonstrated that our adversarial problem
reduces to solving an (uncapacitated) min-cost flow problem. Our numerical results have further illustrated the potential benefit of using the smooth uncertainty set in various applications as an alternative for ellipsoidal and other polyhedral uncertainty sets. Specifically, for the shortest path problem with dependent transit times, 
the solution obtained from using this set performs comparably to the ellipsoidal uncertainty set while yielding significantly easier optimization problems.
Thus, the smooth uncertainty set provides an efficient way to model dependence within the robust optimization framework. While there are many applications that
 could  benefit from 
modeling 
dependence over space, time, or other features of interest, using our approach, 
exploring these applications in greater depth remains a subject for future work.

\section*{Acknowledgement}

Acknowledgements removed for the double-blind review process.

\begin{appendix}
\section{Experiments' Details}\label{appx:experiment_detail}
\shim{\paragraph{Evaluation of  Uncertainty Sets on Energy Data Experiment Details.}
Preprocessing of data included choosing an area of 258 km by 762 km (from the 5255 by 4833 square km pixel grid), and identifying $n=9835$ pixels within this area that had a positive energy consumption between January 2013 and December 2019 (a total of 84 data points for each pixel). The value of each pixel at each month was standardized by subtracting its average and dividing by its standard deviation over the different months.

Given training set $\ell$,
with observations $\{\mp{\tilde\delta}_{\ell,k}\}_{k\in\K_\ell}$.
For each pixel $i$, we computed the mean $\mu_i^{\ell}$, minimum and maximum values
$\bar{\delta}_{\ell,i}=\max_{k\in \mathcal{K_\ell}} \tilde{\delta}_{\ell,k,i}$ and $\ubar{\delta}_{\ell,i}=\min_{k\in \mathcal{K}} \tilde{\delta}_{\ell,k,i}$, as well as the middle and width of the of these values
$\hat{\delta}_{\ell,i}={(\bar{\delta}_{\ell,i}+\ubar{\delta}_{\ell,i})}/{2},\; w^{\ell}_{ii}=\max\{{(\bar{\delta}_{\ell,i}-\ubar{\delta}_{\ell,i})}/{2},1\}.
$
For each pair of pixels $i,j$, we similarly compute the maximal absolute difference~ 
$w_{ij}^{\ell}=\max\{\max_{k\in\mathcal{K}_\ell}|\tilde{\delta}_{\ell,k,i}-\tilde{\delta}_{\ell,k,j}|,0.01\},$ where
the maximization in $w_{ij}$ with 0.01 was done to avoid overfitting created by zero (or very small) 
 pixel differences on the training sets, which do not reflect the true data distribution. Additionally, we compute the sample covariance matrix $\Sigma^\ell$ and its diaginal correction matrix
 $D^{\ell}=\diag(\max\{\vone,\sigma^{\ell}\})$, where $\vone$ is a vector of all ones, and $\sigma^{\ell}$ is the diagonal of $\Sigma^{\ell}$. The maximization between the diagonal $\sigma^\ell$ and 1 is done in order to avoid overfitting by using a standard deviation of at least 1, and for consistency with the definition of $w_{ii}$.

To compare the sizes of the uncertainty sets, we compute an average surrogate for the ellipsoidal parameters $\alpha$ and $\beta$ over the training sets, given by
\begin{align*}\alpha^{\rho,\Omega}_{i}
=\frac{1}{L}\sum_{\ell=1}^L \frac{|\mu^{\ell}_i-\hat{\delta}_{\ell,i}|+\Omega(\tilde{\Sigma}_{\rho}^{\ell})^{-1/2}_{ii}}{w_{ii}^\ell},\quad
\beta^{\rho,\Omega}_{i,j} 
=\frac{1}{L}\sum_{\ell=1}^L \frac{|\mu^{\ell}_i-\mu_j^{\ell}|+\Omega\norm{(\tilde{\Sigma}_{\rho}^{\ell})^{-1/2}e_{ij}}}{w_{ij}^\ell}
,\end{align*}
and their corresponding averages across $i\neq j\in [n]$, $\alpha^{\rho,\Omega}$ and $\beta^{\rho,\Omega}$, respectively.}

\paragraph{Shortest Path Experiment Details}

Suppose that travel time data for $S$ scenarios are given in the form of a matrix $D\in \R_+^{S\times |{\A}|}$ so  {$D_s$ denotes a row for scenario $s\in [S]$} of this matrix. 
For each $a\in \A$, let $\bar D_a$ denote the mean of the $a$th column of $D$ and let $\hat f_a=\bar D_a$. 
Additionally, for each $a\in {\A}$ {and $s\in[S]$, denote the normalized transit times by $\tilde D_{sa}=D_{sa}/\bar{D}_a$, comprising the normalized data matrix $\tilde D$. Also, let  $\hat \delta_a=1$ and let $ \Sigma$ denote the covariance matrix estimated from the normalized data $\tilde D$.}

For the ellipsoidal uncertainty set, the data is used to estimate mean travel times $\hat f$ and corresponding variances. 
In our study illustrated by Figure~\ref{fig:seqejordata}, we used 
a diagonal covariance matrix in order to compute the paths more efficiently in practice while not sacrificing much in terms of performance compared to using the full 
matrix, as observed by~\citep{Chassein2019}. 
A formulation of the robust shortest path with an ellipsoidal uncertainty as an MISOCP problem can be found in~\cite{Chassein2019} and references therein.  
Here (Figure~\ref{fig:seqejordata}) for the ellipsoidal set $\UE$ we experimented with the ellipsoid uncertainty set parameter  $\ngtwo{\sqrt{\Omega}}\in\{1,2,\ldots,11\}$.

For our smooth uncertainty set $\USB$, \ng{$G$ is defined to be the complete graph with vertices $\mathcal A$, and } {we consider two different schemes for setting the parameters of the uncertainty set $\USB$ \ngtwo{(and $\Ubox$ as a special case)} based on the data: 
\begin{itemize}
\item {The \emph{maximum based setting}.} 
For each $a\in {\A}$ and some $\lambda >0$,  we follow~\cite{Chassein2019} and set $\gamma_{aa} = \lambda (\max_{s\in [S]}\{\tilde D_{sa}\}-1)$, 
and for $a'\in \A\setminus\{a\}$, analogously, for some $\lambda'>0$, we set  
\[\gamma_{aa'}=
\lambda'\left(\max_{s\in[S]}\abs{\tilde D_{sa}-\tilde D_{sa'}}-\sum_{s\in [S]}\abs{\tilde D_{sa}-\tilde D_{sa'}}/S\right)+\sum_{s\in [S]}\abs{\tilde D_{sa}-\tilde D_{sa'}}/S.
\] The intuition in setting the bounds $\gamma_{aa'}$ based on the data, here, is that for $\lambda'=0$ the bound corresponds to the mean absolute value of the difference and for $\lambda'=1$ the bound corresponds to the maximum absolute difference, which resembles the setting of the box bounds 
by~\cite{Chassein2019}. \ngtwo{The extension of which to pairs of random variables relates to the Gini mean difference measure~\cite{yitzhaki2003gini}, which is an alternative measure \shim{to} covariance proposed for non-Gaussian random variables.} \ngtwo{We experimented with values of $\lambda\in\{0.025, 0.05, 0.075, \ldots , 0.15, 0.2, \ldots, 0.6\}$ and $\lambda'\in \{0.025, 0.05, 0.075, \ldots , 0.15, 0.2, \ldots, 0.5\}$.}
    \item {The \emph{standard deviation} (st. dev) based setting.}  For each $a\in \A$  
and some $\lambda >0$,  we set $\gamma_{aa} = \lambda \Sigma^{1/2}_{aa}$,  
 and for each $a'\in \A\setminus\{a\}$, and for 
 some $\lambda'>0$, we set  
$\gamma_{aa'}=\lambda' \vectornorm{\Sigma^{1/2}e_{aa'}}$. \ngtwo{We experimented with $\lambda\in\{1,2,\ldots,11\}$.} 
\end{itemize}}

\begin{figure}[t]
    \centering
    \includegraphics[width=0.55\linewidth]{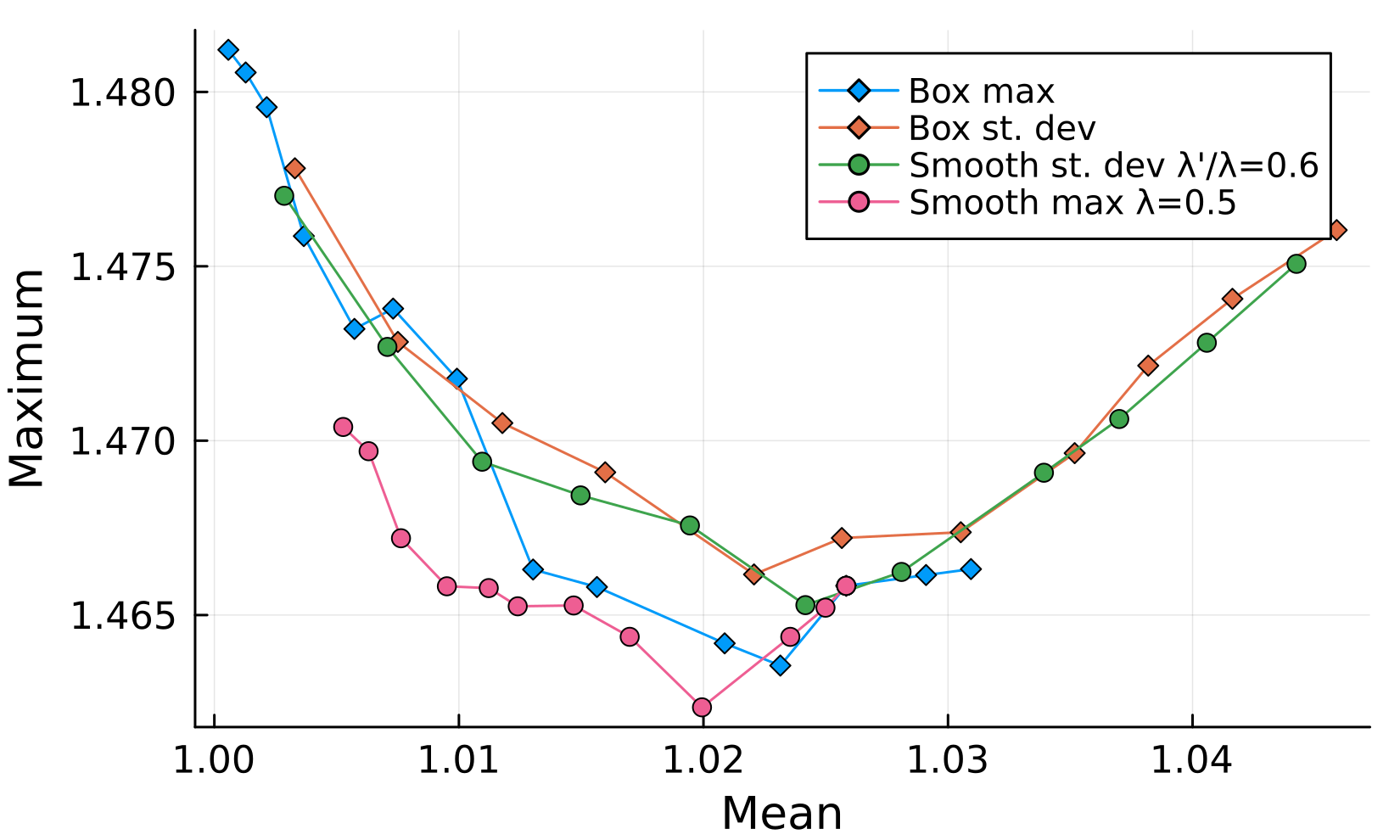}
    \caption{Out-of-sample 
    \shim{trade-off of normalized worst-case and mean path length}
    on the 
    Chicago traffic morning dataset~\cite{Chassein2019} for particular choices of $\lambda'$ and $\lambda$.}
    \label{fig:seqejordataex}
\end{figure}

\ngtwo{Figure~\ref{fig:seqejordataex} displays the smooth uncertainty out-of-sample performance for the max setting with
$\lambda=0.5$ and varied $\lambda'$, and for the standard deviation setting $\lambda'=0.6\lambda$ for varied $\lambda$. The choice of $\lambda'$ in the standard deviation setting did not appear to make \shim{as significant of a difference as 
in the maximum based setting. Overall, the maximum based setting appears to dominate }
the standard deviation based setting. Accordingly, in Figure~\ref{fig:seqejordata} the nondominated points are displayed only for the maximum based setting where both $\lambda$ and $\lambda'$ are varied. The curve with fixed $\lambda=0.5$ in Figure~\ref{fig:seqejordataex} shows that performance similar to that of the nondominated curve  
can be attained by tuning only one of the parameters. In any case, if desired both parameters could be tuned by a cross-validation procedure.}

\end{appendix}

\section*{Code and Data}

The data and code of all our experiments that are a part of this submission will be made available through a github site upon acceptance. 



\bibliographystyle{plainnat}
\bibliography{ref}

\begin{thebibliography}{27}
\providecommand{\natexlab}[1]{#1}
\providecommand{\url}[1]{\texttt{#1}}
\expandafter\ifx\csname urlstyle\endcsname\relax
  \providecommand{\doi}[1]{doi: #1}\else
  \providecommand{\doi}{doi: \begingroup \urlstyle{rm}\Url}\fi

\bibitem[Aissi et~al.(2009)Aissi, Bazgan, and Vanderpooten]{AissiBV09}
Hassene Aissi, Cristina Bazgan, and Daniel Vanderpooten.
\newblock Min-max and min-max regret versions of combinatorial optimization
  problems: {A} survey.
\newblock \emph{Eur. J. Oper. Res.}, 197\penalty0 (2):\penalty0 427--438, 2009.

\bibitem[Atamt{\"u}rk(2006)]{atamturk2006strong}
Alper Atamt{\"u}rk.
\newblock Strong formulations of robust mixed 0--1 programming.
\newblock \emph{Mathematical programming}, 108:\penalty0 235--250, 2006.

\bibitem[B{\"a}rmann et~al.(2016)B{\"a}rmann, Heidt, Martin, Pokutta, and
  Thurner]{barmann2016polyhedral}
Andreas B{\"a}rmann, Andreas Heidt, Alexander Martin, Sebastian Pokutta, and
  Christoph Thurner.
\newblock Polyhedral approximation of ellipsoidal uncertainty sets via extended
  formulations: a computational case study.
\newblock \emph{Computational Management Science}, 13\penalty0 (2):\penalty0
  151--193, 2016.

\bibitem[Belotti et~al.(2013)Belotti, Kirches, Leyffer, Linderoth, Luedtke, and
  Mahajan]{Belotti2013}
Pietro Belotti, Christian Kirches, Sven Leyffer, Jeff Linderoth, James Luedtke,
  and Ashutosh Mahajan.
\newblock Mixed-integer nonlinear optimization.
\newblock \emph{Acta Numerica}, 22:\penalty0 1–131, 2013.

\bibitem[Ben-Tal and Nemirovski(2000)]{ben2000robust}
Aharon Ben-Tal and Arkadi Nemirovski.
\newblock Robust solutions of linear programming problems contaminated with
  uncertain data.
\newblock \emph{Mathematical programming}, 88:\penalty0 411--424, 2000.

\bibitem[Ben-Tal et~al.(2009)Ben-Tal, El~Ghaoui, and Nemirovski]{ben2009robust}
Aharon Ben-Tal, Laurent El~Ghaoui, and Arkadi Nemirovski.
\newblock \emph{Robust optimization}.
\newblock Princeton university press, New Jersey, 2009.

\bibitem[Bertsimas and Sim(2003)]{bertsimas2003robust}
Dimitris Bertsimas and Melvyn Sim.
\newblock Robust discrete optimization and network flows.
\newblock \emph{Mathematical programming}, 98\penalty0 (1-3):\penalty0 49--71,
  2003.

\bibitem[Bertsimas and Sim(2004{\natexlab{a}})]{bertsimas2004}
Dimitris Bertsimas and Melvyn Sim.
\newblock Robust discrete optimization under ellipsoidal uncertainty sets.
\newblock \emph{Manuscript, MIT}, 9, 2004{\natexlab{a}}.

\bibitem[Bertsimas and Sim(2004{\natexlab{b}})]{bertsimas2004price}
Dimitris Bertsimas and Melvyn Sim.
\newblock The price of robustness.
\newblock \emph{Operations research}, 52\penalty0 (1):\penalty0 35--53,
  2004{\natexlab{b}}.

\bibitem[Bertsimas et~al.(2021)Bertsimas, Den~Hertog, and
  Pauphilet]{bertsimas2021probabilistic}
Dimitris Bertsimas, Dick Den~Hertog, and Jean Pauphilet.
\newblock Probabilistic guarantees in robust optimization.
\newblock \emph{SIAM Journal on Optimization}, 31\penalty0 (4):\penalty0
  2893--2920, 2021.

\bibitem[Bertsimas et~al.(2022)Bertsimas, Shtern, and
  Sturt]{bertimas2022two-stage}
Dimitris Bertsimas, Shimrit Shtern, and Bradley Sturt.
\newblock Technical note—two-stage sample robust optimization.
\newblock \emph{Operations Research}, 70\penalty0 (1):\penalty0 624--640, 2022.

\bibitem[Chassein et~al.(2019)Chassein, Dokka, and Goerigk]{Chassein2019}
André Chassein, Trivikram Dokka, and Marc Goerigk.
\newblock Algorithms and uncertainty sets for data-driven robust shortest path
  problems.
\newblock \emph{European Journal of Operational Research}, 274\penalty0
  (2):\penalty0 671--686, 2019.
\newblock ISSN 0377-2217.

\bibitem[Dabadghao et~al.(2025)Dabadghao, Marandi, and
  Roy]{dabadghao2025optimal}
Shaunak~S Dabadghao, Ahmadreza Marandi, and Arkajyoti Roy.
\newblock Optimal interventions in robust optimization with time-dependent
  uncertainties.
\newblock \emph{Computers \& Operations Research}, page 107162, 2025.

\bibitem[Ghaoui et~al.(2003)Ghaoui, Oks, and Oustry]{ghaoui2003worst}
Laurent~El Ghaoui, Maksim Oks, and Francois Oustry.
\newblock Worst-case value-at-risk and robust portfolio optimization: A conic
  programming approach.
\newblock \emph{Operations research}, 51\penalty0 (4):\penalty0 543--556, 2003.

\bibitem[Goldberg et~al.(2025)Goldberg, Langer, and Shtern]{goldberg2024robust}
Noam Goldberg, Mark Langer, and Shimrit Shtern.
\newblock Robust radiotherapy planning with spatially-based uncertainty sets.
\newblock \emph{IISE Transactions}, 57\penalty0 (5):\penalty0 590--606, 2025.

\bibitem[Herer et~al.(2006)Herer, Tzur, and
  Y{\"u}cesan]{herer2006multilocation}
Yale~T Herer, Michal Tzur, and Enver Y{\"u}cesan.
\newblock The multilocation transshipment problem.
\newblock \emph{IIE transactions}, 38\penalty0 (3):\penalty0 185--200, 2006.

\bibitem[Hill et~al.(2015)Hill, Bristow, Fyles, Koritzinsky, Milosevic, and
  Wouters]{hill2015hypoxia}
Richard~P Hill, Robert~G Bristow, Anthony Fyles, Marianne Koritzinsky, Michael
  Milosevic, and Bradly~G Wouters.
\newblock Hypoxia and predicting radiation response.
\newblock \emph{Seminars in radiation oncology}, 25\penalty0 (4):\penalty0
  260--272, 2015.

\bibitem[Khachiyan(1993)]{Khachiyan1993}
Leonid Khachiyan.
\newblock \emph{Complexity of Polytope Volume Computation}, pages 91--101.
\newblock Springer Berlin Heidelberg, Berlin, Heidelberg, 1993.
\newblock ISBN 978-3-642-58043-7.

\bibitem[Kouvelis and Yu(2013)]{kouvelis2013robust}
Panos Kouvelis and Gang Yu.
\newblock \emph{Robust discrete optimization and its applications}, volume~14.
\newblock Springer Science \& Business Media, Dordrecht, 2013.

\bibitem[Mahalanobis(2018)]{mahalanobis2018generalized}
Prasanta~Chandra Mahalanobis.
\newblock On the generalized distance in statistics.
\newblock \emph{Sankhy{\=a}: The Indian Journal of Statistics, Series A
  (2008-)}, 80:\penalty0 S1--S7, 2018.

\bibitem[Navarro(2016)]{navarro2016very}
Jorge Navarro.
\newblock A very simple proof of the multivariate {{Chebyshev's}} inequality.
\newblock \emph{Communications in Statistics-Theory and Methods}, 45\penalty0
  (12):\penalty0 3458--3463, 2016.

\bibitem[Nohadani and Roy(2017)]{Nohadani2017}
Omid Nohadani and Arkajyoti Roy.
\newblock Robust optimization with time-dependent uncertainty in radiation
  therapy.
\newblock \emph{IISE Transactions on Healthcare Systems Engineering},
  7\penalty0 (2):\penalty0 81--92, 2017.

\bibitem[Omer et~al.(2024)Omer, Poss, and Rougier]{omer2023combinatorial}
J\'er\'emy Omer, Michael Poss, and Maxime Rougier.
\newblock Combinatorial {Robust} {Optimization} with {Decision-Dependent}
  {Information} {Discovery} and {Polyhedral} {Uncertainty}.
\newblock \emph{Open Journal of Mathematical Optimization}, 5:\penalty0 5,
  2024.

\bibitem[Orlin(1993)]{Orlin1993}
James~B. Orlin.
\newblock A faster strongly polynomial minimum cost flow algorithm.
\newblock \emph{Operations Research}, 41\penalty0 (2):\penalty0 338--350, 1993.

\bibitem[Xiong et~al.(2018)Xiong, Vahedian, Zhou, Li, and
  Luo]{xiong2018predicting}
Haoyi Xiong, Amin Vahedian, Xun Zhou, Yanhua Li, and Jun Luo.
\newblock Predicting traffic congestion propagation patterns: A propagation
  graph approach.
\newblock In \emph{Proceedings of the 11th ACM SIGSPATIAL International
  Workshop on computational transportation science}, pages 60--69, 2018.

\bibitem[Yan et~al.(2024)Yan, Huang, Ren, Yin, and Qi]{yan2024monthly}
Xiaoqin Yan, Zhou Huang, Shuliang Ren, Ganmin Yin, and Junnan Qi.
\newblock Monthly electricity consumption data at 1 km$\times$ 1 km grid for
  280 cities in china from 2012 to 2019.
\newblock \emph{Scientific Data}, 11\penalty0 (1):\penalty0 877, 2024.

\bibitem[Yitzhaki(2003)]{yitzhaki2003gini}
Shlomo Yitzhaki.
\newblock Gini’s mean difference: A superior measure of variability for
  non-normal distributions.
\newblock \emph{Metron}, 61\penalty0 (2):\penalty0 285--316, 2003.

\end{thebibliography}

\end{document}